\documentclass[12pt]{amsart}
\usepackage{amsmath}
\usepackage{amssymb}
\usepackage{amsthm}
\usepackage[all]{xy}
\usepackage{graphicx}
\usepackage[margin=2.5cm]{geometry}
\usepackage{color}
\usepackage{hyperref}
\usepackage{verbatim}
\usepackage{pb-diagram,pb-xy}

\numberwithin{equation}{section}

\newcommand{\R}{\ensuremath{\mathbb{R}}}
\newcommand{\N}{\ensuremath{\mathbb{N}}}

\newcommand{\Z}{\ensuremath{\mathbb{Z}}}

\newtheorem{theorem}{Theorem}[section]
\newtheorem{corollary}{Corollary}[theorem]
\newtheorem{proposition}[theorem]{Proposition}
\newtheorem{lemma}[theorem]{Lemma}

\theoremstyle{definition} 
\newtheorem{defn}[theorem]{Definition}
\newtheorem{remark}[theorem]{Remark} 

\newtheorem{notation}{Notation}[section]
\newtheorem{Construction}{Construction}[section]

\newcommand{\pr}{{\mathrm{pr}}}

\newcommand{\Ob}{\mathcal{O}b}
\newcommand{\mb}[1]{\mathbf{#1}}

\newcommand{\PH}[1]{\mathbf{P}^{({#1})}}

\DeclareMathOperator*{\Int}{Int}

\DeclareMathOperator*{\Diff}{Diff}

\DeclareMathOperator*{\Emb}{Emb}
\DeclareMathOperator*{\BDiff}{BDiff}

\DeclareMathOperator*{\Hom}{Hom}

\DeclareMathOperator*{\Ab}{Ab}

\DeclareMathOperator*{\kernel}{Ker}
\DeclareMathOperator*{\cokernel}{Coker}

\author{Nathan Perlmutter}

\address{Department of Mathematics, University of Oregon,
Eugene, OR,
  97403, USA} 
  
  \email{nperlmut@uoregon.edu}

\title[Homological Stability For The Moduli Spaces of Products of Spheres]{Homological Stability For The Moduli Spaces of Products of Spheres}

\begin{document}
\maketitle
\begin{abstract}

We prove a homological stability theorem for moduli spaces of
high-dimensional, highly connected manifolds, with respect to forming
the connected sum with the product of spheres $S^{p}\times S^{q}$, for $p
< q < 2p - 2$. This result is analogous to recent results of
S. Galatius and O. Randal-Williams from \cite{GRW 12} and \cite{GRW
  14} regarding the homological stability for the moduli spaces of
manifolds of dimension $2n > 4$, with respect to forming connected
sums with $S^{n}\times S^{n}$.
\end{abstract}

\section{Introduction} \label{Introduction} 
\subsection{Statement of Main Results} \label{subsection: statement of results}
Recently S. Galatius and O. Randal-Williams proved a homological
  stability theorem for moduli spaces of certain manifolds of dimension
  $2n > 4$, with respect to forming connected sums with $S^{n}\times S^{n}$, see \cite{GRW 12} and \cite{GRW 14}.
We prove an analogous homological stability result for the moduli spaces of high-dimensional, highly connected manifolds, with respect to forming
the connected sum with $S^{p}\times S^{q}$, for $p
< q < 2p - 2$.

For what follows, let $M$ be a smooth, $m$-dimensional, compact manifold.  
Denote by $\Diff^{\partial}(M)$ the group of self diffeomorphisms of $M$
which fix some neighborhood of the boundary pointwise, topologized in the $C^{\infty}$-topology.
Our main result concerns the homology of the classifying space $\BDiff^{\partial}(M)$. Following \cite{GRW 14}, the classifying space $\BDiff^{\partial}(M)$ is called the \textit{moduli space of manifolds of type $M$}.

Suppose now that $M$ has non-empty boundary. 
Fix positive integers $p$ and $q$ with $p \leq q$ and $p + q = m$. 
Denote by $V_{p,q}$ the manifold obtained by forming the connected sum of $[0,1]\times\partial M$ with $S^{p}\times S^{q}.$
We then denote by  $M\cup_{\partial M}V_{p,q}$ the manifold obtained by gluing the boundary of $M$ to the cobordism $V_{p,q}$ along $\{0\}\times\partial M \subset V_{p,q}$.
Clearly there is a  diffeomorphism,
$$M\#(S^{p}\times S^{q}) \cong M\cup_{\partial M}V_{p,q}.$$
Consider the homomorphism $\Diff^{\partial}(M) \rightarrow \Diff^{\partial}(M\cup_{\partial M}V_{p,q})$ defined by sending an element $f \in \Diff^{\partial}(M)$ to the diffeomorphism of $M\cup_{\partial M}V_{p,q}$ obtained by extending $f$ identically over $V_{p,q}$.
This homomorphism induces a continuous map on the level of classifying spaces, which we denote by
\begin{equation} \label{eq: classifying space stabilization}
\xymatrix{
s_{p,q}: \BDiff^{\partial}(M) \longrightarrow \BDiff^{\partial}(M\cup_{\partial M}V_{p,q}). 
}
\end{equation}
We refer to (\ref{eq: classifying space stabilization}) as the $(p,q)$-th stabilization map. 
\begin{remark}
In the case that $\partial M$ is not path-connected, the diffeomorphism type of $V_{p,q}$, and hence the homotopy type of the map $s_{p,q}$ will depend on the path component of $\partial M$ that the connected sum is formed in. 
However, our main theorem (Theorem \ref{thm: Main Theorem 1}) holds for any such choice.  
\end{remark}

For $g \in \N$, let
$W^{g}_{p,q}$ 
be the manifold obtained by deleting an $m$-dimensional open disk from the $g$-fold connected-sum, $(S^{p}\times S^{q})^{\# g}$, i.e. $W^{g}_{p,q} = (S^{p}\times S^{q})^{\# g}\setminus\Int(D^{m})$.
Let $r_{p,q}(M)$ be the quantity defined by setting,
$$r_{p,q}(M) := \max\{\; g \in \N \; | \; \text{there exist an embedding $W^{g}_{p,q} \rightarrow M$}\; \}.$$
We refer to the quantity $r_{p,q}(M)$ as the \textit{$(p,q)$-rank} of of the manifold $M$.
In \cite[Theorem 1.2]{GRW 14} S. Galatius and O. Randal-Williams prove that if $M$ is simply-connected, $m = 2n \geq 6$, and $p = q = n$, 
then the map induced by $s_{n, n}$ on integral homology, 
$$\xymatrix{
(s_{n, n})_{*}: H_{k}(\BDiff^{\partial}(M);\; \Z) \longrightarrow H_{k}(\BDiff^{\partial}(M\cup_{\partial M}V_{n, n});\; \Z)
}$$
is an isomorphism if $2k \leq r_{n, n}(M) - 3$ and an epimorphism if $2k \leq r_{n, n}(M) - 1$.

In this paper we extend this homological stability result from \cite{GRW 14} to more general choices of $p$ and $q$ where $p < q$. 
The homotopy group $\pi_{q}(S^{p})$ will play an important role in the proof of our main theorem. 
For any finitely generated abelian group $H$, we let $d(H)$ be the quantity defined by setting,
$$
d(H) := \min\{k \in \N | \; \text{there exists an epimorphism} \; \Z^{\oplus k} \rightarrow H\; \}.
$$
In other words, $d(H)$ is the length of a minimal generating set of $H$. We will refer to $d(H)$ as the \textit{generating set length} of $H$.
With $p$ and $q$ chosen, the quantity $d(\pi_{q}(S^{p}))$ will be a parameter in the statement of our main theorem; it will effect the range of the homological stability result. 
Throughout the paper we will denote by $\kappa(M)$ the \textit{degree of connectivity of $M$}. That is, 
$\kappa(M) := \sup\{\; j \in \N \; | \; \pi_{i}(M) = 0 \; \text{for all $i \leq j$} \; \}.$
For the statement of our main theorem, suppose that the inequalities,
\begin{equation} \label{eq: connectivity condition}
p < q < 2p - 2 \quad \text{and} \quad q - p + 1 < \kappa(M)
\end{equation}
are satisfied. 
\begin{remark}
It is an immediate consequence of the first inequality of (\ref{eq: connectivity condition}) that $p \geq 4$ and that 
$m = p \; + \; q = \dim(M) \geq 9.$ Since $q < 2p - 2$, the \textit{Freudenthal Suspension Theorem} implies that the stabilization map $\pi_{q}(S^{p}) \longrightarrow \pi^{S}_{q-p}$ is an isomorphism. 
It follows that the homotopy group $\pi_{q}(S^{p})$ is a finite abelian group whenever $p < q < 2p - 2$.
\end{remark}
 
\begin{theorem} \label{thm: Main Theorem 1} Let $M$ be a compact manifold with non-empty boundary.
Let $p$ and $q$ be positive integers with $p+ q = \dim(M)$, such that 
$p < q < 2p - 2$  and $q - p + 1 < \kappa(M).$
Let $d$ denote the integer $d(\pi_{q}(S^{p}))$. 
The map,
$$\xymatrix{
(s_{p,q})_{*}: H_{k}(\BDiff^{\partial}(M); \; \Z) \; \longrightarrow \; H_{k}(\BDiff^{\partial}(M\cup_{\partial M}V_{p,q}); \; \Z)
}$$
is an isomorphism if $2k \leq r_{p,q}(M) - 3 - d$ and an epimorphism if $2k \leq r_{p,q}(M) - 1 - d$.
\end{theorem}
A corollary of Theorem \ref{thm: Main Theorem 1} is obtained by setting
$M = W_{p,q}^{g}$
for some $g \in \N$. 
We then have $\partial W_{p,g}^{g} = S^{p+q-1}$ and $V_{p,q} = ([0,1]\times S^{p+q-1})\#(S^{p}\times S^{q})$ and so there is 
a diffeomorphism relative to the boundary,
$W_{p,q}^{g+1} \cong W_{p,q}^{g}\cup_{\partial W^{g}_{p,q}}V_{p,q}.$
Thus (\ref{eq: classifying space stabilization}) yields the map, 
$$\xymatrix{
s_{p,q}: \BDiff^{\partial}(W_{p,q}^{g}) \longrightarrow \BDiff^{\partial}(W_{p,q}^{g+1}).
}
$$
In this case, with $p < q$ as above, we have that $\kappa(W^{g}_{p,q}) = p-1$. With $m = p+q$, condition (\ref{eq: connectivity condition}) then translates to the single inequality $p< q < 2p - 2$. 
We have:
\begin{corollary} \label{cor: corollary 1}
Let $p$ and $q$ be positive integers such that $p < q < 2p - 2$ and let $d$ denote the integer $d(\pi_{q}(S^{p}))$.
Then 
$\xymatrix{
(s_{p,q})_{*}: H_{k}(\BDiff^{\partial}(W^{g}_{p,q});\; \Z) \longrightarrow H_{k}(\BDiff^{\partial}(W^{g+1}_{p, q}); \; \Z)
}$
is an isomorphism when $2k \leq g - 3 - d$ and an epimorphism when $2k \leq g - 1 - d$.
\end{corollary}

One particular case of note is obtained by setting $q = p + 1$. The inequality $p < q < 2p - 2$ is then satisfied when $p \geq 4$. With $p \geq 4$ there is an isomorphism $\pi_{p+1}(S^{p}) \cong \Z/2$ and so $d(\pi_{p+1}(S^{p})) = 1$. We have:
\begin{corollary}
Let $p \geq 4$. Then the homomorphism
$$\xymatrix{
(s_{p,p+1})_{*}: H_{k}(\BDiff^{\partial}(W^{g}_{p, p+1});\; \Z) \longrightarrow H_{k}(\BDiff^{\partial}(W^{g+1}_{p, p+1}); \; \Z)
}$$
is an isomorphism when  $2k \leq g - 4$ and an epimorphism when  $2k \leq g - 2$.
\end{corollary}
In addition to these homological stability results, our techniques yield a ``cancelation'' theorem analogous to \cite[Corollary 6.3]{GRW 14}, for connected sums with products of spheres. 
\begin{theorem} \label{theorem: cancelation}
Let $P$ be a closed $(m-1)$-dimensional manifold and 
let $M$ and $M'$ be $m$-dimensional manifolds equipped with a specified identification, $\partial M = \partial M' = P$.
 Let $p$ and $q$ be positive integers with $p+ q = m$ such that
$p < q < 2p - 2$  and $q - p + 1 < \kappa(M)$. Let $d$ denote the integer $d(\pi_{q}(S^{p}))$. 
Suppose that there is a diffeomorphism $M\#(S^{p}\times S^{q}) \cong M'\#(S^{p}\times S^{q})$ relative to $P$. 
If $r_{p,q}(M) \geq 3 + d$, then there is a diffeomorphism $M \cong M'$ relative to $P$. 
\end{theorem}
\begin{remark}
The above Theorem is also similar to the earlier cancelation result of \cite[Theorem D]{Kr 99} for manifolds of dimension $2n \geq 6$ with respect to connected sums with $S^{n}\times S^{n}$. 
However, Theorem \ref{theorem: cancelation} is proven in a way similar to \cite[Corollary 6.3]{GRW 14} and does not directly use the techniques of \cite{Kr 99}.
\end{remark}

\subsection{Ideas behind the proof}
Our methods are similar to those used in \cite{GRW 12} and \cite{GRW 14}. For positive integers $p$ and $q$ with $p\leq q$ and $p+q = m$, we
construct a highly connected simplicial complex
$K(M)_{p,q}$ which admits an action of the
topological group $\Diff^{\partial}(M)$. Let $W_{p,q}$ denote the manifold $W^{1}_{p,q}$. 
Roughly, the
$l$-simplices of $K(M)_{p,q}$ are given by sets of
$(l+1)$-many pairwise disjoint embeddings
$W_{p, q}  \hookrightarrow M,$
with a pre-prescribed boundary condition. This simplicial complex
is similar to the complexes constructed in \cite{GRW 12} and \cite{GRW 14} and in fact, in the case that $p = q$ it is the same simplicial complex. 
Furthermore we use many of the simplicial techniques developed in those papers to study $K(M)_{p,q}$.
The majority of the technical mathematical work of this paper is involved in proving that if the inequalities of (\ref{eq: connectivity condition}) are satisfied, then the geometric realization $|K(M)_{p,q}|$ is $\frac{1}{2}[r_{p,q}(M) - 4 - d(\pi_{q}(S^{p}))]$-connected. This is established in Theorem \ref{thm: high connectivity of K(M)} and is proven in Section \ref{section: high connectivity of K} using techniques developed over the course of Sections \ref{section: simplicial complexes} through \ref{section: connectivity of L}. The $7$th and final section shows how Theorem \ref{thm: high connectivity of K(M)} implies Theorem \ref{thm: Main Theorem 1}.

To prove that $K(M)_{p,q}$ is highly connected when (\ref{eq: connectivity condition}) is satisfied, we compare $K(M)_{p,q}$ to an auxiliary complex constructed out of algebraic data associated to the manifold $M$, based on certain structures and invariants developed in \cite{W 63} by C.T.C. Wall. These structures and invariants from \cite{W 63} that we use to define the auxiliary complex are as follows: 
\begin{itemize} \itemsep8pt
\item A bilinear map
$\tau_{p,q}: \pi_{p}(M)\otimes\pi_{q}(S^{p}) \longrightarrow \pi_{q}(M)$
defined by $([\varphi], [f]) \; \mapsto \; [\varphi\circ f].$

\item Bilinear pairings,
$
\lambda_{p, q}: \pi_{p}(M)\!\otimes\!\pi_{q}(M) \longrightarrow \Z$ and
$\mu_{q}: \pi_{q}(M)\!\otimes\!\pi_{q}(M) \longrightarrow \pi_{q}(S^{p})$
where $\mu_{q}$ is $(-1)^{q}$-symmetric.

\item Functions $\alpha_{i}: \pi_{i}(M) \longrightarrow \pi_{i-1}(SO_{m-i})$ for $i = p, q,$
which are defined by sending an element $x \in \pi_{i}(M)$ to the class in $\pi_{i-1}(SO_{m-i})$ that classifies the normal bundle of an embedding which represents $x$. (If (\ref{eq: connectivity condition}) is satisfied, then by \cite[Proposition 1 and Lemma 1]{W 63} any such class $x \in \pi_{i}(M)$ is represented by a smooth embedding, unique up to regular homotopy).  
\end{itemize}

We consider the algebraic structure given by the data of the pair of homotopy groups $\pi_{p}(M)$, $\pi_{q}(M)$ equipped with the maps,
$\tau_{p,q}, \; \lambda_{p,q}, \; \mu_{q}, \; \alpha_{p}, \; \alpha_{q}$
given above. We call this algebraic structure the \textit{Wall form of degree $(p,q)$} associated to $M$ and denote it by $\mathcal{W}_{p,q}(M)$. 
We construct a simplicial complex $L(\mathcal{W}_{p,q}(M))$ whose $l$-simplices are given by sets of $(l+1)$-many
pairwise orthogonal (with respect to both $\lambda_{p,q}$ and $\mu_{q}$), algebraic embeddings $\mathcal{W}_{p,q}(W_{p,q}) \rightarrow \mathcal{W}_{p,q}(M),$
mimicking embeddings of the
  manifolds $W_{p,q} \hookrightarrow M$. 
 
  In Section \ref{section: connectivity of L} we prove that 
   that $|L(\mathcal{W}_{p,q}(M))|$ is $\frac{1}{2}[r_{p,q}(M) - 4 - d(\pi_{q}(S^{p}))]$-connected and in Section \ref{section: high connectivity of K} we prove that $|K(M)_{p,q}|$ is  $\frac{1}{2}[r_{p,q}(M) - 4 - d(\pi_{q}(S^{p}))]$-connected by comparing the complex $K(M)_{p,q}$ to $L(\mathcal{W}_{p,q}(M))$. 
   The cancelation result, Theorem \ref{theorem: cancelation} follows from Proposition \ref{proposition: cancelation} as a consequence of the fact that $|L(\mathcal{W}_{p,q}(M))|$ is path-connected when $r_{p,q}(M) \geq 4 - d$. 

\subsection{Acknowledgements} The author would like to thank Boris Botvinnik for suggesting this particular problem and for numerous helpful discussions on the subject of this paper. The author is also thankful to Allen Hatcher for some helpful remarks regarding the exposition and to the anonymous referee for very helpful critical comments and suggestions regarding the mathematics.

\section{Simplicial Complexes} \label{section: simplicial complexes}
\subsection{The Complex of Embedded Handles} \label{subsection: embedded handles}
For what follows, let $p$ and $q$ be positive integers with $p \leq q$. Recall the manifolds, 
$W^{g}_{p,q} = (S^{p}\times S^{q})^{\# g}\setminus\text{Int}(D^{p+q}).$ 
We let $W_{p,q}$ denote the manifold $W^{1}_{p,q}$.
We construct a simplicial complex similar to the one
from \cite[Definition 4.1]{GRW 12} (in fact, in the case that $p = q$ it is exactly the same complex).  For the construction we will need a slight modification of the manifolds $W_{p,q}$. Fix an oriented embedding,
$$\beta: \{1\}\times D^{p+q-1} \longrightarrow \partial W_{p,q}.$$
We define $\bar{W}_{p,q}$ to be the manifold obtained from $W_{p,q}$ by attaching $[0, 1]\times D^{p+q-1}$ to $W_{p,q}$ along the embedding $\beta$, i.e.\ $\bar{W}_{p,q} := W_{p,q}\cup_{\beta}(D^{p+q-1}\times[0,1]).$

We construct a \textit{core} $C_{p,q}
\subset \bar{W}_{p,q}$ as follows. Choose a base point $(a_{0}, b_{0}) \in S^{p}\times S^{q}$ such that:
\begin{enumerate}
\item[(a)] the subspace $(S^{p}\times \{b_{0}\})\cup (\{a_{0}\}\times S^{q}) \subset S^{p}\times S^{q}$ is contained in $W_{p,q} \subset S^{p}\times S^{q}$,
\item[(b)] the pair $(a_{0}, -b_{0}) \in S^{p}\times S^{q}$ is contained in $W_{p,q}$. 
\end{enumerate}
We then choose an embedded path $\gamma$ in $\bar{W}_{p,q} = W_{
  p,q}\cup_{\beta}(D^{p+q-1}\times[0,1])$ from the point
$(a_{0}, -b_{0})$ to $(0,0) \in [0,1]\times D^{p+q-1}$ whose interior
does not intersect
$
(S^{p}\times \{a_{0}\})\cup (\{b_{0}\}\times
S^{q})
$
and whose image agrees with $[0,1]\times \{0\} 
\subset [0,1]\times D^{p+q-1}$
inside $[0,1]\times D^{p+q-1}$.

We then define $C_{p,q}$ to be the subspace of $\bar{W}_{p,q}$
given by
$$
C_{p,q} := (S^{p}\times \{a_{0}\})\cup (\{b_{0}\}\times S^{q})\cup \gamma([0,1])\cup (\{0\}\times D^{p+q-1})
\; \subset \; \bar{W}_{p,q}.
$$
It is immediate that $C_{p,q}$ is homotopy equivalent to $S^{p}\vee S^{q}$ and that $C_{p,q}$ is a deformation retract of $\bar{W}_{p,q}$. 

Now let $M$ be a compact manifold of dimension $m$ with non-empty boundary. 
Choose an embedding of a coordinate patch,
$a: [0,1)\times \R^{m-1} \longrightarrow M$
such that $a^{-1}(\partial M) =  \{0\}\times\R^{m-1}$.  For each pair of positive integers $p \leq q$ with $p+q = m$,
we define a simplicial complex,
  $K(M, a)_{p,q}$. 
\begin{defn} \label{defn: the embedding complex} Let $M$ and $a: [0,1)\times\R^{m-1} \longrightarrow M$ be as above. The simplicial complex $K(M, a)_{p,q}$ is defined as follows:
\begin{enumerate} 
\item[(i)] 
A vertex in $K(M, a)_{p,q}$ is defined to be a pair $(t, \phi)$, where $t \in \R$ and 
$\phi: \bar{W}_{p,q} \rightarrow M$ is an embedding for which there exists $\epsilon > 0$ such that for
$(s, z) \in [0, \epsilon)\times D^{m-1} \; \subset \; \bar{W}_{p,q}$, \; 
the equality $\phi(s, z) = a(s, z + te_{1})$ is satisfied,
where $e_{1} \in \R^{m-1}$ denotes the first basis vector. 
\item[(ii)] A set of vertices $\{(\phi_{0}, t_{0}), \dots, (\phi_{l}, t_{l})\}$ forms an $l$-simplex 
if 
$t_{i} \neq t_{j}$ and \\
$\phi_{i}(C_{p,q}) \cap \phi_{j}(C_{p,q}) = \emptyset$ whenever $i \neq j$. 
\end{enumerate}
\end{defn}
\begin{remark}
The definition of the complex $K(M, a)_{p,q}$ depends on the choice of
$a: [0,1)\times\R^{m-1} \rightarrow M.$
However this embedding never plays a role and all of our results regarding $K(M, a)_{p,q}$ will hold for all choices of $a$.  
So, we
will suppress $a$ from the notation and let 
$K(M)_{p,q}$ denote the complex $K(M, a)_{p,q}.$ 
\end{remark}

With $M$ still as above, suppose now that $p$ and $q$ are such that $p + q = m$ and 
$$p < q < 2p - 2 \quad \text{and} \quad q - p + 1 < \kappa(M)
$$
as in (\ref{eq: connectivity condition}) (recall $\kappa(M)$ is the connectivity of $M$). 
The majority of the technical work of this paper is geared toward proving the following result. 
\begin{theorem} \label{thm: high connectivity of K(M)} Let $M$, $p$, and $q$ be as above and let $d$ denote the integer
$d(\pi_{q}(S^{p}))$, the generating set length of $\pi_{q}(S^{p})$. Let $g \geq 0$ be such that $r_{p,q}(M) \geq g$. Then the geometric realization $|K(M)_{p,q}|$ is $\frac{1}{2}(g-4-d)$-connected.
\end{theorem} 

The next proposition is a generalization of \cite[Corollary 4.5]{GRW 12} and is proven in the same way. 
Combining Theorem \ref{thm: high connectivity of K(M)} together with the next proposition implies the cancelation result Theorem \ref{theorem: cancelation} stated in the introduction. 
\begin{proposition} \label{proposition: cancelation} 
Let $P$ be a closed manifold of dimension $m-1$ and let $M$ and $M'$ be 
compact manifolds of dimension $m$ equipped with a specified identification, $\partial M = \partial M' = P$.
Let $p \leq q$ be positive integers such that $p+q = m$. 
Suppose that there is a diffeomorphism $\varphi: M\#(S^{p}\times S^{q}) \stackrel{\cong} \longrightarrow M'\#(S^{p}\times S^{q})$ relative $P$. 
If $|K(M\#(S^{p}\times S^{q}))_{p,q}|$ is connected, then there is a diffeomorphism $M \cong M'$ relative $P$. 
\end{proposition}

\subsection{Some Simplicial Techniques}
We will need to use some general
techniques and results pertaining to simplicial complexes. In this section we state the relevant results
needed. Recall that for a simplicial complex $X$, the \textit{link} of a
simplex $\sigma < X$ is defined to be the sub-simplicial complex
of $X$ consisting of all simplices that are adjacent to $\sigma$
but which do not occur as a face of $\sigma$.
We denote the link of the simplex $\sigma$ by $Lk_{X}(\sigma).$

\begin{defn} A simplicial complex $X$ is said to be
\textit{weakly Cohen-Macaulay} of dimension $n$ if it is
$n-1$-connected and the link of any $l$-simplex is
$(n-l-2)$-connected. In this case we write $\omega CM(X) \geq
n$. The complex $X$ is said to be \textit{locally weakly
Cohen-Macaulay} of dimension $n$ if the link of any simplex is
$(n - l- 2)$-connected (but no global connectivity is required on
$X$ itself). In this case we shall write $lCM(X) \geq n$.
\end{defn}

We will need to use the important following two results from \cite{GRW 14}. Theorem \ref{thm: cohen mac trick} is a generalization of the ``Coloring Lemma'' of Hatcher and Wahl from \cite[Lemma 3.1]{HW 10}.
\begin{theorem} \label{thm: cohen mac trick}
Let $X$ be a
simplicial complex with $lCM(X) \geq n$, $f: \partial I^{n}
\rightarrow |X|$ be a map which is simplicial with respect to
some PL triangulation of $\partial I^{n}$, and $h: I^{n}
\rightarrow |X|$ be a null-homotopy of $f$. Then the triangulation
extends to a PL triangulation of $I^{n}$, and $h$ is homotopic
relative to $\partial I^{n}$ to a simplicial map $g: I^{n}
\rightarrow |X|$ with the property that $g$ is simplex-wise injective on the interior, $I^{n}\setminus\partial
I^{n}$.
\end{theorem}

\begin{proposition} \label{prop: high connected map} Let $X$ be
a simplicial complex, and $Y \subset X$ be a full subcomplex.
Let $n$ be an integer with the property that for each
$l$-simplex $\sigma < X$, the complex $Y\cap Lk_{X}(\sigma)$ is
$(n-l-1)$-connected. Then the inclusion $|Y| \hookrightarrow
|X|$ is $n$-connected.
\end{proposition}

\section{The Algebraic Invariants} \label{Section: The Algebraic Invariants}

\subsection{Intersection Products and Vector Bundles} \label{Intersection Products and Vector Bundles}   
In this section we implement some of the 
invariants developed by C.T.C. Wall in \cite{W 63}
associated to smooth manifolds in a certain dimensional and connectivity range. 
We change notation and rephrase some of these algebraic structures from \cite{W 63}, however most of the necessary constructions and proofs can be found in \cite{W 63} and so we omit most proofs and refer the reader there for details.

For what follows, let $M$ be a compact, simply connected, oriented manifold of dimension $m$.
As in the introduction, we denote by $\kappa(M)$ the connectivity of $M$.
Since $\kappa(M) \geq 1$ by assumption, it follows that $\pi_{s}(M)$ is in bijection with the set of homotopy classes of unbased maps $S^{s} \rightarrow M$ when $s > 1$. 
We will need to represent such maps by embeddings.
Recall that a homotopy $f_{t}: S^{s} \rightarrow M$ is said to be an \textit{isotopy} if $f_{t}$ is an embedding for all $t \in [0,1]$ and is said to be a \textit{regular homotopy} if $f_{t}$ is an immersion for all $t \in [0,1]$. 
 By combining the results \cite[Proposition 1]{W 63} and \cite[Lemma 1]{W 63}
we have the following:
\begin{lemma} \label{lemma: represent by embeddings}
Suppose that $s \leq \min\{\; \tfrac{m + \kappa(M) - 2}{2}, \; \tfrac{2m - 3}{3} \; \}$. 
Then any map $S^{s} \longrightarrow M$ is homotopic to an embedding. 
Furthermore any two embeddings $S^{s} \longrightarrow M$ which are homotopic as continuous maps are regularly homotopic as immersions. 
\end{lemma}
\begin{remark} \label{remark: regular homotopy}
The result \cite[Proposition 1]{W 63} also implies that if $s \leq \min\{\; \tfrac{m + \kappa(M) - 2}{2}, \; \tfrac{2m - 4}{3} \; \}$ then any two embeddings $S^{s} \longrightarrow M$ which are homotopic as maps are actually isotopic as embeddings.
For our purposes however, uniqueness up to regular homotopy is all that will be needed. 
\end{remark}

For what follows let $p$ and $q$ be positive integers with $p+ q = m = \dim(M)$,  such that
\begin{equation} \label{eq: primary inequality}
p < q < 2p - 2 \quad \text{and} \quad q - p + 1 < \kappa(M)
\end{equation}
is satisfied. 
By Lemma \ref{lemma: represent by embeddings} every element in the groups $\pi_{p}(M)$ and $\pi_{q}(M)$ can be represented by an embedding $S^{s} \rightarrow M$ for $s = p, q$ which is unique up to regular homotopy.
We proceed to construct extra structure associated to the homotopy groups $\pi_{p}(M)$ and $\pi_{q}(M)$.

First, 
we define a bilinear map
\begin{equation} \label{eq: tau map}
\tau_{p,q}: \pi_{p}(M)\otimes\pi_{q}(S^{p}) \longrightarrow \pi_{q}(M), \quad ([f], [g]) \mapsto [f\circ g].
\end{equation}
Since $q < 2p -2$, the \textit{Freudenthal Suspension Theorem} implies that the suspension map, $\Sigma: \pi_{q-1}(S^{p-1}) \longrightarrow \pi_{q}(S^{p})$ is surjective. 
 It follows from this fact that (\ref{eq: tau map}) is indeed bilinear. 
 
 Next, we have a $(-1)^{q}$-symmetric, bilinear pairing 
\begin{equation} \label{eq: mu pairing}
\mu_{q}: \pi_{q}(M)\otimes\pi_{q}(M) \longrightarrow \pi_{q}(S^{p}).
\end{equation}
The definition of $\mu_{q}$ can be found in \cite[Page 255]{W 63} (in \cite{W 63} the map $\mu_{q}$ is actually denoted by $\lambda_{q,q}$). 
Below, in Construction \ref{construction: pontryagin thom} we review the definition of $\mu_{q}(x, y)$ for $x, y \in \pi_{q}(M)$.
\begin{Construction} \label{construction: pontryagin thom}
Let $x, y \in \pi_{q}(M)$ and let $g: S^{q} \rightarrow M$ be an embedding which represents the class $y$. 
\begin{enumerate}
\item[i.] Let $D^{q}_{-}, D^{q}_{+} \subset S^{q}$ denote the upper and lower hemisphere of $S^{q}$ respectively.
Let $U_{\pm}$ denote the submanifold $g(D^{q}_{\pm}) \subset M$ and let $V_{-}$ be a closed tubular neighborhood of $U_{-}$ in $M$ such that $U_{+}\cap V_{-} = \partial U_{+}$. 
Then let $X$ denote the manifold $M\setminus\Int(V_{-})$. 
We have $(U_{+}, \partial U_{+}) \subset (X, \partial X)$. 
\item[ii.] Since $q \leq m - 2$ and $V_{-}$ is diffeomorphic to a disk, it follows that the homomorphism $i_{*}: \pi_{q}(X) \rightarrow \pi_{q}(M)$ induced by the inclusion $i: X \hookrightarrow M$ is an isomorphism. Let $f: S^{q} \longrightarrow M$ be a map which represents the class $i_{*}^{-1}(x) \in \pi_{q}(X)$. 
\item[iii.] Denote by $N \subset X$ the normal disk-bundle of $U_{+}$ embedded in $X$ as a closed tubular neighborhood. 
Since $U_{+}$ is contractible, the normal bundle $N \rightarrow U_{+}$ is trivial. 
Let $\phi: N \stackrel{\cong} \longrightarrow U_{+}\times D^{p}$ be a trivialization that is consistent with the orientation on $N$ induced by the orientations on $X$ and on $U_{+}$. 
Since $U_{+}$ is contractible there is only one such trivialization up to homotopy. 
 \item[iv.] Let $T_{g}: X \longrightarrow D^{p}/\partial D^{p} = S^{p}$ be the \textit{Pontryagin-Thom collapse map}, defined by 
 $$T_{g}(x) = 
 \begin{cases}
 [\text{proj}_{D^{p}}(\phi(x))]  &\quad \text{if $x \in N,$}\\
 [\partial D^{p}] &\quad \text{if $x \notin N$.}
 \end{cases}
 $$
 \item[v.] Consider the composition, $S^{q} \stackrel{f} \longrightarrow X \stackrel{T_{g}} \longrightarrow S^{p}$ where recall, $f: S^{p} \rightarrow X$ is a map which represents the element $i_{*}^{-1}(x) \in \pi_{q}(X)$ chosen in step ii.
 The class $\mu_{q}(x, y)$ is then defined by setting, $\mu_{q}(x, y) := [T_{g}\circ f] \in \pi_{q}(S^{p}).$
 \end{enumerate}
 \end{Construction}
 It is immediate from the above construction that $\mu_{q}$ is linear in the first variable. 
 It is not immediate that $\mu_{q}$ is linear in the second variable or that $\mu_{q}$ is $(-1)^{q}$-symmetric, i.e. $\mu_{q}(x, y) = (-1)^{q}\cdot\mu_{q}(y, x)$. 
 The proof that $\mu_{q}$ is bilinear and $(-1)^{q}$-symmetric is covered by the corollary to \cite[Lemma 4]{W 63} on Page 256. 
 \begin{remark}
In Construction \ref{construction: pontryagin thom} the class $[T_{g}\circ f] \in \pi_{q}(S^{p})$ depends on the isotopy class of the embedding $g$ and the homotopy class of the map $f$.
 It then follows from the $(-1)^{q}$-symmetry of $\mu_{q}$ that $[T_{g}\circ f] = \mu_{q}([f], [g])$ 
 depends only on the homotopy class of $g$ and not the isotopy class. 
 It follows that $\mu_{q}$ is well defined even in the case when elements of $\pi_{q}(M)$ are represented by embeddings unique up to regular homotopy, but not necessarily unique up to isotopy. 
 \end{remark}
 
There is a bilinear pairing between $\pi_{p}(M)$ and $\pi_{q}(M)$ defined as follows. 
First denote by \\
$h_{i}: \pi_{i}(M) \longrightarrow H_{i}(M; \Z)$ the \textit{Hurewicz map}
and by $\Lambda_{p,q}: H_{p}(M; \Z)\otimes H_{q}(M; \Z) \longrightarrow M$ the \textit{homological intersection pairing} (since $M$ is oriented by assumption, the homological intersection pairing is defined). 
We then define 
\begin{equation} \label{eq: intersection pairing lambda}
\lambda_{p,q}: \pi_{p}(M)\otimes\pi_{q}(M) \longrightarrow \Z
\end{equation}
to be the bilinear map given by the formula, 
$$\lambda_{p,q}(x, y) := \Lambda_{p,q}(h_{p}(x), \; h_{q}(y)) \quad \text{for $x \in \pi_{p}(M)$ and $y \in \pi_{q}(M)$}.$$ 
It follows easily that if $f: S^{p} \rightarrow M$ and $g: S^{q} \rightarrow M$ are transversal embeddings that represent $x \in \pi_{p}(M)$ and $y \in \pi_{q}(M)$ respectively,
then $\lambda_{p,q}(x, y)$ is equal to the (oriented) \textit{algebraic intersection number} associated to $f(S^{p})\cap g(S^{q})$. 

The following proposition describes the relationship between $\tau_{p,q}$, $\mu_{q}$, and $\lambda_{p,q}$. 
The proof of the following proposition follows immediately from the construction of $\mu_{q}$ and $\lambda_{p,q}$ and is given in \cite[Page 255]{W 63}.
\begin{proposition} \label{proposition: intersection and pre-comp}
Let $M$ be an oriented manifold of dimension $m$ and let $p$ and $q$ be positive integers with $p + q = m$ such that 
$p < q < 2p - 2$ and $q - p + 1 < \kappa(M)$. 
The equations 
$$\lambda_{p, q}(x, \; \tau_{p,q}(x', z)) = 0 \quad \text{and} \quad  \mu_{q}(\tau_{p,q}(x,  z), \; y) \; = \; \lambda_{p,q}(x, \; y)\cdot z$$
are satisfied  for all $x, x' \in \pi_{p}(M)$, $y \in \pi_{q}(M)$, and $z \in \pi_{q}(S^{p})$.
\end{proposition}
\begin{remark}
In \cite{W 63}, $\lambda_{p,q}$ is defined as a bilinear map over $\pi_{p}(S^{p})$. The second equation in the above proposition in \cite{W 63} reads as $\mu_{q}(\tau_{p,q}(x,  z), \; y) = \lambda_{p,q}(x, y)\circ z$, where $z \in \pi_{q}(S^{p})$ and $\lambda_{p,q}(x, y) \in \pi_{p}(S^{p})$. 
To identify our formula with that of C.T.C. Wall from \cite{W 63}, we are implicitly using the fact that post-composition of an element $z \in \pi_{q}(S^{p})$  by a degree $l$ map $S^{q} \rightarrow S^{q}$, induces multiplication by $l$ when the element $z$ is in the stable range. 
\end{remark}

We will also need to consider functions
\begin{equation} \label{eq: normal bundle map}
\alpha_{q}: \pi_{q}(M) \longrightarrow \pi_{q-1}(SO_{p}) \quad \text{and} \quad \alpha_{p}: \pi_{p}(M) \longrightarrow \pi_{p-1}(SO_{q})
\end{equation}
which are defined by sending an element $x \in \pi_{s}(M)$ for $s = p, q$, to the class in $\pi_{s-1}(SO_{m-s})$ which classifies the normal bundle of an embedding $S^{s} \rightarrow M$ which represents $x$. These maps are well defined since by Lemma \ref{lemma: represent by embeddings}, any element in $\pi_{s}(M)$ for $s = p, q$ is represented by an embedding which is unique up to regular homotopy (recall that the isomorphism class of the normal bundle of any immersion $S^{s} \rightarrow M$ is an invariant of its regular homotopy class). 
To describe the relationship between $\alpha_{p}$, $\alpha_{q}$, and $\mu_{q}$, we must consider maps,
\begin{equation} \label{eq: form parameter maps}
\partial_{q}: \pi_{q}(S^{p}) \longrightarrow \pi_{q-1}(SO_{p}) \quad \text{and} \quad  \bar{\pi}_{q}: \pi_{q-1}(SO_{p}) \longrightarrow \pi_{q}(S^{p}). 
\end{equation}
The map $\partial_{q}$ is the boundary homomorphism in the long exact sequence in homotopy groups associated to the fibre sequence $SO_{p} \rightarrow SO_{p+1} \rightarrow S^{p}$ and $\bar{\pi}_{q}$ is defined by the composition 
$\pi_{q-1}(SO_{p}) \rightarrow \pi_{q-1}(S^{p-1}) \rightarrow \pi_{q}(S^{p})$
where the first map is the homomorphism induced by the projection $SO_{p} \rightarrow SO_{p}/SO_{p-1} \cong S^{p-1}$ and the second is the suspension homomorphism. 
The next proposition follows from \cite[Theorem 1]{W 63}. 
\begin{proposition} \label{prop: intersection twist} 
Let $M$ be an oriented manifold of dimension $m$ and let $p$ and $q$ be positive integers with $p + q = m$ such that 
$p < q < 2p - 2$ and $q - p + 1 < \kappa(M)$. 
Then the equations 
\begin{itemize} \itemsep8pt
\item $\alpha_{p}(x + x') = \alpha_{p}(x) + \alpha_{p}(x')$,
\item $\alpha_{q}(y + y') = \alpha_{q}(y) + \alpha_{q}(y') + \partial_{q}(\mu_{q}(y, y'))$, 
\item $\mu_{q}(y, y) = \bar{\pi}_{q}(\alpha_{q}(y))$. 
\end{itemize}
are satisfied for all $x, x' \in \pi_{p}(M)$, $y, y' \in \pi_{q}(M)$.
\end{proposition}

We will also need to understand the relationship between $\alpha_{p}$, $\alpha_{q}$, and $\tau_{p,q}$. 
We will need to consider a bilinear map
\begin{equation} \label{eq: bundle composition}
 F_{p,q}: \pi_{p-1}(SO_{q})\otimes\pi_{q}(S^{p}) \longrightarrow \pi_{q-1}(SO_{p})
\end{equation}
which is defined in \cite[Lemma 2]{W 63} as follows. For $(x, z) \in \pi_{p-1}(SO_{q})\times\pi_{q}(S^{p})$, let $E(x) \rightarrow S^{p}$ be the vector bundle with fibre-dimension $q$, classified by the element $x \in \pi_{p-1}(SO_{q})$. Choose a section $\sigma: S^{p} \rightarrow E(x)$ of the bundle. 
Since the dimension of the total-space $E(x)$ is $p+q$, it follows from Lemma \ref{lemma: represent by embeddings} that the map $\alpha_{q}: \pi_{q}(E(x)) \rightarrow \pi_{q-1}(SO_{p})$ is well defined.
The map $F_{p,q}$ is then defined by setting,
$F_{p,q}(x, z) = \alpha_{q}(\sigma_{*}(z)),$ where $\sigma_{*}: \pi_{q}(S^{p}) \rightarrow \pi_{q}(E(x))$ is the map induced by $\sigma$.
Since any two sections of $E(x) \rightarrow S^{p}$ are fibrewise homotopic, it follows that the above formula used to define $F_{p,q}(x, z)$ does not depend on the choice of section $\sigma$. 
We have the following proposition.
\begin{proposition} \label{prop: bilinearity of F}
Suppose that $p < q < 2p - 2$. 
Then the map given by $(x, y) \mapsto F_{p,q}(x, y)$, as defined above is bilinear.
\end{proposition}
\begin{proof} Since $q < 2p - 2$, the Freudenthal Suspension Theorem implies that the suspension map $\Sigma: \pi_{q-1}(S^{p-1}) \rightarrow \pi_{q}(S^{p})$ is surjective. 
It then follows directly from \cite[Lemma 5 (1)]{W 63} that $F_{p, q}$ is linear in the first variable. 
We now must prove that $F_{p,q}$ is linear in the second variable. 
Let $y, y' \in \pi_{q}(S^{p})$ and $x \in \pi_{p-1}(SO_{q})$. 
Let $E(x) \rightarrow S^{p}$ be a vector bundle with fibre-dimension $q$, which represents the class $x \in \pi_{p-1}(SO_{q})$. 
Choose a section $\sigma: S^{p} \rightarrow E(x)$ and denote by $z \in \pi_{p}(E(x))$ the class represented by $\sigma$.
By the definition of the bilinear map $\tau_{p,q}: \pi_{p}(E(x))\otimes\pi_{q}(S^{p}) \longrightarrow \pi_{q}(E(x))$ from (\ref{eq: tau map}), we have 
$\sigma_{*}(y) = \tau_{p,q}(z, y)$ and $\sigma_{*}(y') = \tau_{p,q}(z, y')$. 
By Proposition \ref{proposition: intersection and pre-comp} it follows that,
$$\mu_{q}(\sigma_{*}(y), \sigma_{*}(y')) \; = \; \mu_{q}(\tau_{p,q}(z, y), \tau_{p,q}(z, y')) \; = \; \lambda_{p,q}(z, \tau_{p,q}(z, y'))\cdot y \; = \; 0,$$
and thus by Proposition \ref{prop: intersection twist} we have,
$$
\alpha_{q}(\sigma_{*}(y) + \sigma_{*}(y')) \; = \; \alpha_{q}(\sigma_{*}(y)) + \alpha_{q}(\sigma_{*}(y')).
$$
It then follows from the definition of $F_{p,q}$ that,
\begin{equation} \label{eq: alpha linearity 2}
F_{p,q}(x, y + y') \; = \; \alpha_{q}(\sigma_{*}(y) + \sigma_{*}(y')) \; = \; \alpha_{q}(\sigma_{*}(y)) + \alpha_{q}(\sigma_{*}(y')) \; = \; F_{p,q}(x, y) + F_{p,q}(x, y').
\end{equation}
This proves that $F_{p,q}$ is linear in the second variable. 
This concludes the proof of the proposition.
\end{proof}
With $F_{p,q}$ defined and shown to be bilinear when $p < q < 2q - 2$, we may now describe the relationship between $\alpha_{p}$, $\alpha_{q}$, and $\tau_{p,q}$. 
\begin{proposition} \label{prop: relationship between alpha, tau} Let $M$ be a manifold of dimension $m$ and let $p$ and $q$ be positive integers with $p + q = m$ such that  $p < q < 2p - 2$ and $q - p + 1 < \kappa(M)$. 
Then the equation 
$$\alpha_{q}(\tau_{p,q}(x, z)) \; = \; F_{p,q}(\alpha_{p}(x), \; z)$$ 
is satisfied for all $x \in \pi_{p}(M)$ and $z \in \pi_{q}(S^{p})$.
\end{proposition}
\begin{proof} 
Let $x \in \pi_{p}(M)$ and $z \in \pi_{q}(S^{p})$. 
Let $g: S^{p} \rightarrow M$ be an embedding that represents $x$ and let $f: S^{q} \rightarrow S^{p}$ be a map that represents $z$. 
Let $E \subset M$ be a tubular neighborhood of $g(S^{p}) \subset M$ and let $\sigma: S^{p} \rightarrow E$ be a section of the projection map $E \rightarrow S^{p}$, which is a vector bundle whose isomorphism class is determined by $\alpha_{p}(x)$. 
Let $f': S^{q} \rightarrow E$ be an embedding homotopic to the composition $\sigma\circ f$. 
Denote by $i_{E}: E \hookrightarrow M$ the inclusion map.
Notice that, 
$(i_{E})_{*}(\sigma_{*}(z)) = \tau_{p, q}(x, z)$ as classes of $\pi_{q}(M)$.
Since $\alpha_{p}(x)$ classifies the bundle $E \rightarrow S^{p}$, it follows that 
$$F_{p,q}(\alpha_{p}(x), z) = \alpha_{q}(\sigma_{*}(z)) = \alpha_{q}(i_{*}\circ\sigma_{*}(z)) = \alpha_{q}(\tau_{p,q}(x, z)),$$
where the second equality follows from the fact that $\alpha_{q}$ is natural with respect to co-dimension $0$ embeddings.
This concludes the proof of the proposition.
\end{proof}

We sum up the results from this section with the following lemma.
\begin{lemma} \label{lemma: properties of the invariants}
Let $M$, $p$, and $q$ be as above. The maps 
\begin{itemize} \itemsep8pt
\item $\tau_{p,q}: \pi_{p}(M)\otimes\pi_{q}(S^{p}) \longrightarrow \pi_{q}(M)$,
\item $\lambda_{p, q}: \pi_{p}(M)\otimes\pi_{q}(M) \longrightarrow \Z$,
\item $\mu_{q}: \pi_{q}(M)\otimes\pi_{q}(M) \longrightarrow \pi_{q}(S^{p})$,
\item $\alpha_{p}: \pi_{p}(M) \longrightarrow \pi_{p-1}(SO_{q})$,
\item $\alpha_{q}: \pi_{q}(M) \longrightarrow \pi_{q-1}(SO_{p})$,
\end{itemize}
satisfy the following conditions. For all $x,  x' \in \pi_{p}(M)$, $y,  y' \in \pi_{q}(M)$ and $z \in \pi_{q}(S^{p})$ we have:
\begin{enumerate} \itemsep8pt
\item[i.] 
$\lambda_{p, q}(x, \; \tau_{p,q}(x', z)) = 0,$
\item[ii.] 
$\mu_{q}(\tau_{p,q}(x,  z), \; y) \; = \; \lambda_{p,q}(x, \; y)\cdot z,$
\item[iii.] 
$\alpha_{q}(x + x') = \alpha_{p}(x) + \alpha_{p}(x'),$
\item[iv.]
$\alpha_{q}(y + y') \; = \; \alpha_{q}(y) + \alpha_{q}(y') + \partial_{q}(\mu_{q}(y, y')),$
\item[v.]
$\mu_{q}(y, y) = \bar{\pi}_{q}(\alpha_{q}(y)),$
\item[vi.]  
$\alpha_{q}(\tau_{p,q}(x, z)) \; = \; F_{p,q}(\alpha_{p}(x), \; z).$
\end{enumerate}
\end{lemma}

The next proposition follows immediately from the construction of $\tau_{p,q}$, $\lambda_{p,q}$, $\mu_{q}$, $\alpha_{p}$ and $\alpha_{q}$.\begin{proposition} \label{prop: functoriality of the structure} Let $M$, $p$, and $q$ be as above and let $N$ be an oriented manifold of dimension $m = \dim(M)$ such that $\kappa(N) > q - p + 1$ is satisfied. 
Let $\psi: N \longrightarrow M$ be a smooth embedding. Then the maps on homotopy groups of degree $p$ and $q$ induced by $\psi$ preserves all values of $\tau_{p,q}$, $\lambda_{p,q}$, $\mu_{q}$, $\alpha_{p}$ and $\alpha_{q}$.
\end{proposition}

\subsection{Modifying Intersections} \label{subsection: modifying intersections}
We will ultimately need to use $\lambda_{p, q}$ and $\mu_{q}$ to study intersections of embedded submanifolds.
Let $M$ be an oriented manifold of dimension $m$ and let $p$ and $q$ be positive integers such that $p + q = m$.
The integer $\lambda_{p,q}(x, y)$ is equal to the signed intersection number associated to embeddings $f: S^{p} \rightarrow M$ and $g: S^{q} \rightarrow M$ which represent $x$ and $y$ respectively. 
If $M$ is simply-connected, then by application of the \textit{Whitney trick} \cite[Theorem 6.6]{M 65}, one can deform $g$ through a smooth isotopy to a new embedding $g': S^{q} \rightarrow M$ such that $g'(S^{q})$ and $f(S^{p})$ intersect transversally at exactly $|\lambda_{p,q}(x,y)|$-many points. Now 
let $f, g: S^{q} \rightarrow M$ be
embeddings whose images intersect transversally. 
We will need a higher dimensional analogue of the Whitney trick that applies to the intersection of such embeddings.
In \cite{We 67} and \cite{HQ 74},
such generalized versions of the Whitney trick are
developed. We have:
\begin{proposition} {\rm (Wells, \cite{We 67})} \label{prop: Generalized Whitney Trick} 
Let $M$, $p$ and $q$ be as above and suppose that $p < q < 2p - 2$ and  $q - p + 1 < \kappa(M)$.
Let
$f, g: S^{q} \longrightarrow M$ be
embeddings. Then there
exists an isotopy $\Psi_{t}: S^{q} \rightarrow M$ such that
 $\Psi_{0} = g$
\text{and} $\Psi_{1}(S^{q})\cap f(S^{q}) = \emptyset$, if and only if $\mu_{q}([g], [f]) = 0$. Furthermore,
if $f$ and $g$ are transverse and $U \subset S^{q}$ is an open disk
containing $g^{-1}(f(S^{q}))$, then the isotopy 
$\Psi_{t}$ may be chosen so that $\Psi_{t}(x) = g(x)$ for all $x \in S^{q}\setminus U$ and $t \in [0,1]$.
\end{proposition}
\begin{remark} \label{remark: relative diffeotopy}
 This is a specialization of the main theorem of \cite{We 67}. In \cite{We 67} the author associates to submanifolds $A^{r}, B^{s} \subset X^{m}$ with $\partial A, \partial B \subset \partial X$ and $\partial A\cap\partial B = \emptyset$, an invariant $\alpha(A, B; X)$ valued in the stable homotopy group
$\pi^{S}_{r+s-m}$
and proves that if certain dimensional and connectivity conditions are
met, then there exists an isotopy $\Psi_{t}: A \rightarrow X$ such that
$\Psi_{0} = i_{A}$, $\Psi_{t}|_{\partial_{A}} = i_{\partial A}$ for all $t \in [0,1]$, and $\Psi_{1}(A) \cap B = \emptyset$, 
if and only if
$\alpha(A, B; X) = 0$, where $i_{A}$ and $i_{\partial A}$ are the inclusion maps.   
By letting $A$, $B$, and $X$ denote $g(S^{q})$, $f(S^{q})$, and $X$ respectively, 
it follows from the construction of $\alpha(A, B; X)$ in \cite{We 67} that $\alpha(A, B; X)$ is equal to $\mu_{q}([g], [f])$ after identifying $\pi_{q}(S^{p})$ with $\pi_{2q-m}^{S}$. 

To obtain the second part of Proposition \ref{prop: Generalized Whitney Trick} we do the following. 
Suppose that $f$ and $g$ are transversal and let $D^{q}_{+} \subset S^{q}$ be a closed disk which contains $g^{-1}(f(S^{q}))$ (we may assume that this disk is the upper hemisphere, hence the notation). 
Let $D^{q}_{-} \subset S^{q}$ be the closure of the complement of $D^{q}_{+}$ (which is the lower hemisphere) and let $U_{\pm}$ denote the submanifold $g(D^{q}_{\pm})\subset M$. 
Let $N \subset M$ be a closed tubular neighborhood of $U_{-}$ such that $N\cap f(S^{q}) = \emptyset$ and  $N\cap U_{+} = \partial U_{+}$. 
We then let $A$, $B$, and $X$ denote $U_{+}$, $f(S^{q})$, and $M\setminus\Int(N)$ respectively. 
It follows from the construction of $\alpha(A, B; X)$ in \cite{W 63} that 
$$\alpha(A, B; X) =  \alpha(g(S^{q}), f(S^{q}); M).$$
If $\mu_{q}([g], [f]) = \alpha(g(S^{q}), f(S^{q}); M) = 0$, then  
the relative statement of the main theorem from \cite{We 67} (described in the previous paragraph) applied to $A, B$ and $X$, 
implies that there exists an isotopy $\psi_{t}: A \rightarrow X$ such that
$\psi_{0} = i_{A}$, $\psi_{t}|_{\partial_{A}} = i_{\partial A}$ for all $t \in [0,1]$, and $\psi_{1}(A) \cap B = \emptyset$.
Now, let $A'$ denote the submanifold $g(S^{q}) \subset M$. 
We define $\Psi_{t}: A' \rightarrow M$ to be  
 the isotopy obtained by setting $\Psi_{t}$ equal to $\psi_{t}$ on $A$ and equal to $i_{A'}$ on $A' \setminus A$ for all $t \in [0,1]$. 
 This yields the second statement in Proposition \ref{prop: Generalized Whitney Trick}.
 \end{remark}

\subsection{Application to Products of Spheres} \label{subsection: Application to Products of Spheres}
We will now determine the values of the maps from Lemma \ref{lemma: properties of the invariants} on the homotopy groups of the manifolds $W^{g}_{p,q}$, for $p < q < 2p-2$.
Recall the notation $W_{p,q} = W^{1}_{p,q}$.
Choose base points $a_{0} \in S^{p}$ and $b_{0} \in S^{q}$ so that $(a_{0}, b_{0}) \in W_{p,q} \subset S^{p}\times S^{q}$. 
Let
\begin{equation} \label{eq: standard inclusions}
i_{p}: S^{p} \longrightarrow W_{p,q} \quad \text{and} \quad i_{q}: S^{q} \longrightarrow W_{p,q}
\end{equation}
be the maps given by the inclusions,
$S^{p}\times\{b_{0}\} \hookrightarrow W_{p,q}$ and $\{a_{0}\}\times S^{q} \hookrightarrow W_{p,q}.$
Let $\sigma \in  \pi_{p}(W_{p,q})$ and $\rho \in \pi_{q}(W_{p,q})$ be the elements represented by the maps from (\ref{eq: standard inclusions}).
The inclusion map $W_{p,q} \hookrightarrow S^{p}\times S^{q}$ induces isomorphisms,
$$\begin{aligned}
&\pi_{p}(W_{p,q}) \cong \pi_{p}(S^{p}\times S^{q}) \cong \pi_{p}(S^{p}),\\
&\pi_{q}(W_{p,q}) \cong \pi_{q}(S^{p}\times S^{q}) \cong \pi_{q}(S^{q})\oplus\pi_{q}(S^{p}).
 \end{aligned}$$
Denote by $\widehat{\tau}_{p,q}: \pi_{q}(S^{p}) \rightarrow \pi_{q}(W_{p,q})$ the map given by $z \mapsto \tau_{p,q}(\sigma, z)$. 
It follows from the above isomorphisms that $\widehat{\tau}_{p,q}$ is a monomorphism. 
From this we have:
\begin{proposition} The $p$th and $q$th homotopy groups of $W_{p,q}$ are as follows:
$$\pi_{p}(W_{p,q}) = \Z\langle \sigma \rangle, \quad \pi_{q}(W_{p,q}) = \Z\langle \rho \rangle\oplus \widehat{\tau}_{p,q}(\pi_{q}(S^{p})).$$
\end{proposition}
The embeddings $i_{p}: S^{p} \rightarrow W_{p,q}$ and $i_{q}: S^{q} \rightarrow W_{p,q}$ both have trivial normal bundles and $i_{p}(S^{p})$ and $i_{q}(S^{q})$ intersect transversally at exactly one point. It follows that 
$$\alpha_{p}(\sigma) = 0, \quad \alpha_{q}(\rho) = 0, \quad \text{and} \quad \lambda_{p,q}(\sigma, \rho) = 1.$$
The boundary-connected-sum decomposition, $W^{g}_{p,q} = (W_{p,q})^{\natural g}$
induces isomorphisms,
\begin{equation} \label{eq: direct sum decomp}
\pi_{p}(W^{g}_{p,q}) \stackrel{\cong} \longrightarrow  \pi_{p}(W_{p,q})^{\oplus g}, \quad \pi_{q}(W^{g}_{p,q})\stackrel{\cong} \longrightarrow \pi_{q}(W_{p,q})^{\oplus g}.
\end{equation}
For $i = 1, \dots, g$, denote by $\sigma_{i}$ and $\rho_{i}$ the elements of $\pi_{p}(W^{g}_{p,q})$ and $\pi_{q}(W^{g}_{p,q})$ which correspond to the elements $\sigma$ and $\rho$ coming from the $i$th direct summand under the isomorphisms (\ref{eq: direct sum decomp}). The boundary connect sum decomposition $W^{g}_{p,q} \cong (W_{p,q})^{\natural g}$ together with Proposition \ref{lemma: properties of the invariants} yields:
\begin{proposition} \label{eq: values of the invariants 2} The maps $\lambda_{p,q}$, $\mu_{q}$, $\alpha_{q}$, and $\alpha_{p}$ take the following values on the manifold $W^{g}_{p,q}$. For all $z \in \pi_{q}(S^{p})$ and $i, j = 1, \dots, g$:
$$\xymatrix@C-.5pc@R-1.5pc{
\lambda_{p,q}(\sigma_{i}, \rho_{j}) = \delta_{i,j},  & \mu_{q}(\tau_{p,q}(\sigma_{i}, z), \; \rho_{j}) = \delta_{i,j}\cdot z, & \mu_{q}(\rho_{i}, \rho_{j}) = 0, \\
\alpha_{p}(\sigma_{i}) = 0, &\alpha_{q}(\rho_{i}) = 0, &\alpha_{q}(\tau_{p, q}(\sigma_{i}, z)) = 0.
}
$$
\end{proposition}

\section{Wall Forms} \label{section: Algebra}
We now formalize the algebraic structures studied in the previous section and give an axiomatic development of the \textit{Wall form} introduced in the introduction. 
\subsection{Abelian $H$-Pairs} \label{Abelian H Pairs}
We first must describe the 
category underlying all of our constructions. 
Fix a finitely generated Abelian group $H$. An object $\mb{M}$ in the category $\Ab^{2}_{H}$ is defined to be a pair of abelian groups $(\mb{M}_{-}, \mb{M}_{+})$ equipped with a bilinear map, 
$\tau: \mb{M}_{-}\otimes H \longrightarrow \mb{M}_{+}.$
We call such an object an \textit{$H$-pair} and refer to $\Ab^{2}_{H}$ as the category of $H$-pairs. We call the abelian groups $\mb{M}_{\pm}$ the \textit{component} groups of $\mb{M}$.
A morphism $f: \mb{M} \rightarrow \mb{N}$ of $H$-pairs 
 is defined to be a pair of group homomorphisms 
 $$f_{-}: \mb{M}_{-} \longrightarrow \mb{N}_{-}, \quad f_{+}: \mb{M}_{+} \longrightarrow \mb{N}_{+}$$ 
which make the diagram,
\begin{equation} \label{eq: morphism diagram}
\xymatrix{
\mb{M}_{-}\otimes H \ar[rr]^{f_{-}\otimes Id_{H}} \ar[d]^{\tau_{\mb{M}}} && \mb{N}_{-}\otimes H \ar[d]^{\tau_{\mb{N}}} \\
\mb{M}_{+} \ar[rr]^{f_{+}} && \mb{N}_{+}
}
\end{equation}
commute. We will refer to morphisms in $\Ab^{2}_{H}$ as \textit{$H$-maps}.
\begin{notation}
We will always denote the bilinear map associated to an $H$-pair by $\tau$. If more than one $H$-pair is present, say $\mb{M}$ and $\mb{N}$, we will decorate the maps with the subscripts $\mb{M}$, $\mb{N}$, i.e.\ $\tau_{\mb{M}}$ and $\tau_{\mb{N}}$ as in diagram (\ref{eq: morphism diagram}), so that there is no ambiguity. 
\end{notation}

 The category $\Ab^{2}_{H}$ is easily seen to be an additive category. Kernels, co-kernels, direct-sums, and direct-products are all formed component-wise. Direct-sums are equal to direct-products when taken over finite families of objects and thus the operation of binary direct-sum yields a bi-product. It is also immediate that for any two $H$-pairs $\mb{M}$ and $\mb{N}$, $\Hom_{\Ab_{H}^{2}}(\mb{M}, \mb{N})$ has the structure of an abelian group via pointwise addition of $H$-maps. 
 
 An $H$-pair $\mb{N}$ is said to be a sub-$H$-pair of $\mb{M}$ if $\mb{N}_{\pm} \leq \mb{M}_{\pm}$ are sub-groups and $\tau_{\mb{N}}$  is equal to the restriction of $\tau_{\mb{M}}$ to $\mb{N}_{-}\otimes H$. In this case we denote $\mb{N} \leq \mb{M}$. If $\mb{N}^{1}, \mb{N}^{2} \leq \mb{M}$ are both sub-$H$-pairs then we denote by $\mb{N}^{1}\cap\mb{N}^{2}$ the sub-$H$-pair of $\mb{M}$ defined by setting $(\mb{N}^{1}\cap\mb{N}^{2})_{\pm} := \mb{N}^{1}_{\pm}\cap\mb{N}^{2}_{\pm}$.
 
There are two basic $H$-pairs that we will use to ``probe'' other $H$-pairs. We define $\mathbf{P}^{(0)} \in \Ob(\Ab_{H}^{2})$ by setting
\begin{equation} \label{eq: Z-0}
\PH{0}_{-} := 0, \quad \PH{0}_{+} := \Z, \quad \tau = 0,
\end{equation}
and we define $\PH{1} \in \Ob(\Ab_{H}^{2})$ by setting
\begin{equation} \label{eq: Z-1}
\PH{1}_{-} := \Z, \quad \PH{1}_{+} := H, \quad \tau(t, h) = t\cdot h.
\end{equation}
It is easy to verify that $\PH{0}$ and $\PH{1}$ are \textit{projective} objects in $\Ab_{H}^{2}$ and that every finitely generated projective $H$-pair is a direct sum of copies of $\PH{0}$ and $\PH{1}$, although we will not explicitly need this fact.

The following general proposition about $H$-pairs will be useful to us latter on. 
\begin{proposition} \label{prop: direct sum iso}
Let $\mb{M}$ and $\mb{N}$ be $H$-pairs. 
Suppose that there is an isomorphism of $H$-pairs, $\mb{M} \cong \mb{M}\oplus\mb{N}$ and that $\mb{M}$ is finitely generated. 
Then $\mb{N} = \mb{0}$.
\end{proposition}
\begin{proof}
From the isomorphism $\mb{M} \cong \mb{M}\oplus\mb{N}$, we have an isomorphism of abelian groups, $\mb{M}_{-} \cong \mb{M}_{-}\oplus\mb{N}_{-}$. 
Since $\mb{M}_{-}$ is a finitely generated abelian group, it follows that $\mb{N}_{-} = 0$. 
Since $\mb{N}_{-} = 0$, it follows that $\kernel(\tau_{\mb{N}}) = 0$ and $\cokernel(\tau_{\mb{N}}) = \mb{N}_{+}$. 
Now, notice that if $\mb{M} = \mb{A}\oplus\mb{B}$ is a direct-sum decomposition of $H$-pairs, then there are abelian group isomorphisms,
$$
\kernel(\tau_{\mb{M}}) \cong \kernel(\tau_{\mb{A}})\oplus \kernel(\tau_{\mb{B}}) \quad \text{and} \quad \cokernel(\tau_{\mb{M}}) \cong \cokernel(\tau_{\mb{A}})\oplus\cokernel(\tau_{\mb{B}}).
$$
 Applying the above isomorphisms to the direct-sum decomposition $\mb{M} = \mb{M}\oplus\mb{N}$, using the fact that $\kernel(\tau_{\mb{N}}) = 0$ and $\cokernel(\tau_{\mb{N}}) = \mb{N}_{+}$,  we obtain an abelian group isomorphism, 
 $$\cokernel(\tau_{\mb{M}}) \cong \cokernel(\tau_{\mb{M}})\oplus\mb{N}_{+}.$$
 Since $\cokernel(\tau_{\mb{M}}) \leq \mb{M}_{+}$ is a finitely generated abelian group, it follows that $\mb{N}_{+} = 0$. 
 This concludes the proof. 
\end{proof}

\subsection{Wall forms} \label{subsection: Wall forms} We fix once and for all a finitely generated abelian group $H$. 
All of our constructions will take place in the category $\Ab^{2}_{H}$. 
Let $\mb{G}$ be a finitely generated abelian $H$-pair equipped with group-homomorphisms,
$\partial: H \longrightarrow \mb{G}_{+}$ and $\pi: \mb{G}_{+} \longrightarrow H.$
Then, let $\epsilon = \pm 1$.
We call such a $4$-tuple $(\mb{G}, \partial, \pi, \epsilon)$ a \textit{form-parameter} for the category $\Ab^{2}_{H}$. Now, fix a form-parameter $(\mb{G}, \partial, \pi, \epsilon)$ 
and let $\mb{M}$ be a finitely generated $H$-pair.
Consider the following:
\begin{itemize} \itemsep8pt
\item A bilinear map, $\lambda: \mb{M}_{-}\otimes \mb{M}_{+} \longrightarrow \Z$.
\item An $\epsilon$-symmetric bilinear form, $\mu: \mb{M}_{+}\otimes\mb{M}_{+} \longrightarrow H$.
\item Functions, $\alpha_{\pm}: \mb{M}_{\pm} \longrightarrow \mb{G}_{\pm}$.
\end{itemize}
\begin{defn} \label{defn: Wall-form} The $5$-tuple $(\mb{M}, \lambda, \mu, \alpha_{-}, \alpha_{+})$ is said to be a \textit{Wall form} with parameter $(\mb{G}, \partial, \pi, \epsilon)$ if the following conditions are satisfied.
For all $x, \; x' \in \mb{M}_{-}$, \; $y, \; y' \in \mb{M}_{+}$, and $h \in H$:
\begin{enumerate} \itemsep8pt
\item[i.] $\lambda(x, \tau_{\mb{M}}(x', h)) = 0$, 
\item[ii.] 
$\mu(\tau_{\mb{M}}(x, h), y) \; = \; \lambda(x, y)\cdot h, \; $
\item[iii.]
$\alpha_{-}(x + x') = \alpha_{-}(x) + \alpha_{-}(x'),$
\item[iv.] 
$\alpha_{+}(y + y') = \alpha_{+}(y) + \alpha_{+}(y') + \partial(\mu(y, y')), $
\item[v.]
$\mu(y, y) \; = \; \pi(\alpha_{+}(y)),$
\item[vi.] 
$\alpha_{+}(\tau_{\mb{M}}(x, h)) \; = \; \tau_{\mb{G}}(\alpha_{-}(x), h).$
\end{enumerate}
\end{defn}

Note that condition iii. implies that $\alpha_{-}$ is a group homomorphism. Condition iv.\  however implies that $\alpha_{+}$ is in general not a homomorphism and so $(\alpha_{-}, \alpha_{+})$ is not an $H$-map. 

\begin{notation} We will often denote a Wall form by its underlying $H$-pair, i.e.\ $\mathbf{M} := (\mb{M}, \lambda, \mu, \alpha_{-}, \alpha_{+})$. We will always use the same notation to denote the associated maps, $\lambda, \mu, \alpha_{-}, \alpha_{+}$. If another Wall form is present, say $\mathbf{N}$, we will decorate the associated maps with subscripts and denote, $\lambda_{\mb{N}}, \mu_{\mb{N}}, \alpha_{\mb{N}  -}, \alpha_{\mb{N}  +}$ so that there is no ambiguity. 
  \end{notation}

\begin{defn}
A \textit{morphism} between Wall forms (with the same form-parameter) is an  $H$-map $f: \mb{M} \longrightarrow \mb{N}$ 
 that preserves all values of $\lambda$, $\mu$, and $\alpha_{\pm}$.  
 \end{defn}
 
 \begin{remark}
 It will be important to distinguish between maps of $H$-pairs and morphisms of Wall forms. 
 We will always refer to maps in the category $\Ab_{H}^{2}$ as \textit{$H$-maps} and will reserve the word \textit{morphism} for morphisms of Wall forms. We will only consider morphisms between Wall forms with the same form-parameter.
\end{remark}

If  $(\mb{M}, \lambda, \mu, \alpha_{-}, \alpha_{+})$ is a Wall form and $\mb{N} \leq \mb{M}$ is a sub-$H$-pair, then the Wall form structure on $\mb{M}$ induces a unique Wall form structure on $\mb{N}$ by restricting the maps $\lambda, \mu, \alpha_{-}, \alpha_{+}$ to $\mb{N}$. All of the required conditions from Definition \ref{defn: Wall-form} automatically hold. With this induced Wall form structure, the inclusion $H$-map $\mb{N} \hookrightarrow \mb{M}$ is a morphism of Wall forms. In this way we are justified in calling $\mb{N}$ a \textit{sub-Wall form} of $\mb{M}$. 
 
 Recall from Section \ref{Abelian H Pairs} the $H$-pairs $\PH{0}$ and $\PH{1}$. For any Wall form $\mb{M}$ there are ``duality'' homomorphisms,
\begin{equation} \label{eq: duality maps}
\xymatrix{
T^{(0)}: \mb{M}_{-} \longrightarrow \Hom_{\Ab_{H}^{2}}(\mb{M}, \PH{0}), \quad T^{(1)}: \mb{M}_{+} \longrightarrow \Hom_{\Ab_{H}^{2}}(\mb{M}, \PH{1})
}
\end{equation}
defined as follows. For $v \in \mb{M}_{-}$ we set,
$$T^{(0)}_{-}(v)(x) = 0 \in \mb{0} = \PH{0}_{-},  \quad T^{(0)}_{+}(v)(y) = \lambda(v, y) \in \Z = \PH{0}_{+}, \quad \text{for $x \in \mb{M}_{-}$, $y \in \mb{M}_{+}$.}$$ 
Then, for $w \in \mb{M}_{+}$ we set,
$$T^{(1)}_{-}(w)(x) = \lambda(x, w) = \Z\ \in \PH{1}_{-},  \quad T^{(1)}_{+}(w)(y) = \mu(y, w) \in H = \PH{1}_{+}, \quad \text{for $x \in \mb{M}_{-}$, $y \in \mb{M}_{+}$.}$$ 

\begin{defn} \label{defn: non-singular} A Wall form $\mb{M}$ is said to be \textit{non-singular} if the duality homomorphisms $T^{(0)}$ and $T^{(1)}$ from (\ref{eq: duality maps}) are both isomorphisms. 
\end{defn}

We will need notation for orthogonal complements. 
Let $\mb{N} \leq \mb{M}$. 
We define a new sub-Wall form $\mb{N}^{\perp} \leq \mb{M}$ by setting:
\begin{equation} \label{eq: orthog comp}
\begin{aligned}
\mb{N}^{\perp}_{-} \; &:= \; \{\; x \in \mb{M}_{-} \; | \; \lambda(x, w) = 0 \; \text{for all $w \in \mb{N}_{+}$} \},\\
\mb{N}^{\perp}_{+} \; &:= \: \{\; y \in \mb{M}_{+} \; | \; \lambda(v, y) = 0 \; \text{and} \; \mu(y, w) = 0 \; \text{for all $v \in \mb{N}_{-}$,  $w \in \mb{N}_{+}$}\}.
\end{aligned}
\end{equation}
It can be easily checked that $\tau(\mb{N}^{\perp}_{-}\otimes H) \leq \mb{N}^{\perp}_{+}$ and thus $\mb{N}^{\perp}$ actually is a sub-$H$-pair of $\mb{M}$. 
We call $\mb{N}^{\perp}$ the \textit{orthogonal complement} to $\mb{N}$ in $\mb{M}$. 

\begin{defn} \label{defn: orthogonal sub-wall forms}
Two sub-Wall forms $\mb{N}, \mb{N}' \leq \mb{M}$ are said to be \textit{orthogonal} if 
the following two conditions are met:
\begin{enumerate}
\item[i.]  $\mb{N}\cap\mb{N}' = \mb{0}$,
\item[ii.] $\mb{N} \leq (\mb{N}')^{\perp}$  and $\mb{N}' \leq \mb{N}^{\perp}$.
\end{enumerate}
\end{defn}

\begin{proposition} \label{prop: orthog direct sum}
Let $\mb{N}$ and $\mb{M}$ be Wall forms. There is a unique Wall form structure on the $H$-pair $\mb{N}\oplus\mb{M}$ such that the natural inclusions $\mb{N}, \mb{M} \hookrightarrow \mb{N}\oplus\mb{M}$ are morphisms and such that $\mb{N}, \mb{M} \leq \mb{N}\oplus\mb{M}$ are orthogonal as sub-Wall forms. 
\end{proposition}
\begin{proof}  
We define $\mu_{\mb{N}\oplus\mb{M}}, \lambda_{\mb{N}\oplus\mb{M}},$ and $(\alpha_{\mb{N}\oplus\mb{M}})_{-}$ by setting,
$$\mu_{\mb{N}\oplus\mb{M}} := \mu_{\mb{N}}\oplus\mu_{\mb{M}}, \quad \lambda_{\mb{N}\oplus\mb{M}} := \lambda_{\mb{N}}\oplus\lambda_{\mb{M}}, \quad \text{and} \quad (\alpha_{\mb{N}\oplus\mb{M}})_{-} := (\alpha_{\mb{N}})_{-}\oplus (\alpha_{\mb{M}})_{-}.$$
Since $\mu$ and $\lambda$ are bilinear and $\alpha_{-}$ is a homomorphism according to Definition \ref{defn: Wall-form},
it makes sense to form the direct-sums of maps as above and so the above formulas are well defined.
To define $(\alpha_{\mb{N}\oplus\mb{M}})_{+}$ (which will not be a homomorphism), we set
$$(\alpha_{\mb{N}\oplus\mb{M}})_{+}(x + y) = (\alpha_{\mb{N}})_{+}(x) + (\alpha_{\mb{M}})_{+}(y) \quad \text{for $x \in \mb{N}_{+}$ and $y \in \mb{M}_{+}$.}$$
Since $\mu_{\mb{N}\oplus\mb{M}}(x, y) = 0$ for $x \in \mb{N}_{+} \leq \mb{N}_{+}\oplus\mb{M}_{+}$ and $y \in \mb{M}_{+} \leq \mb{N}_{+}\oplus\mb{M}_{+}$, it follows that $(\alpha_{\mb{N}\oplus\mb{M}})_{+}$ satisfies all of the required conditions from Definition \ref{defn: Wall-form}, thus we have a Wall form structure on $\mb{N}\oplus\mb{M}$. 
It is immediate that the natural inclusions $\mb{N}, \mb{M} \hookrightarrow \mb{N}\oplus\mb{M}$ are morphisms and that $\mb{N}, \mb{M} \leq \mb{N}\oplus\mb{M}$ are orthogonal. 
Uniqueness follows from the fact that the maps $\mu_{\mb{N}\oplus\mb{M}}, \lambda_{\mb{N}\oplus\mb{M}},$ and $(\alpha_{\mb{N}\oplus\mb{M}})_{\pm}$ as defined above, are completely determined by their values on $\mb{N} \leq \mb{N}\oplus\mb{M}$ and $\mb{M} \leq \mb{N}\oplus\mb{M}$. 
\end{proof}
\begin{notation} \label{notation: orthogonal sum notation}  Given two Wall forms $\mb{N}$ and $\mb{M}$, the direct-sum $\mb{N}\oplus\mb{M}$ will always be assumed to have the Wall form structure described in Proposition \ref{prop: orthog direct sum}. 
If $\mb{N}, \mb{N}' \leq \mb{M}$ are sub-Wall forms which are orthogonal according to Definition \ref{defn: orthogonal sub-wall forms}, then we will denote by $\mb{N}\perp\mb{N}'$ the internal sum of the sub-Wall forms $\mb{N}$ and $\mb{N}'$ in $\mb{M}$.
We will continue to use the symbol $\oplus$ to denote the direct sums of $H$-pairs and abelian groups. 
The meaning of $\oplus$ will be made clear by the context in which it is used.
\end{notation}

\subsection{The Wall form Associated to a Manifold} \label{subsection: Wall-form of a manifold}
Wall forms will arise for us in the following way. Let $M$ be a compact, oriented, $m$-dimensional manifold. Let $p$ and $q$ be positive integers with $p + q = m$ such that $p < q < 2p - 2$ and $q - p - 1 < \kappa(M)$ is satisfied. 
We set $H = \pi_{q}(S^{p})$ and denote by, $\mathcal{W}_{p,q}(M)$ the $\pi_{q}(S^{p})$-pair given by the data,
$$\mathcal{W}_{p,q}(M)_{-} := \pi_{p}(M), \quad \mathcal{W}_{p,q}(M)_{+} := \pi_{q}(M), \quad \tau:= \tau_{p,q}: \pi_{p}(M)\otimes\pi_{q}(S^{p}) \longrightarrow \pi_{q}(M)$$
where $\tau_{p,q}: \pi_{p}(M)\times\pi_{q}(S^{p}) \longrightarrow \pi_{q}(M)$ is the bilinear map
from (\ref{eq: tau map}). We need to define a suitable form-parameter for $\Ab^{2}_{\pi_{q}(S^{p})}$. 
Denote by $\mathbf{G}_{p,q}$ the $\pi_{q}(S^{p})$-pair given by the data,
$$(\mathbf{G}_{p,q})_{-} := \pi_{p-\!1}(SO_{q}), \quad (\mathbf{G}_{p,q})_{+} := \pi_{q-\!1}(SO_{p}), \quad \tau := F_{p,q}: \pi_{p-\!1}(SO_{q})\otimes \pi_{q}(S^{p}) \longrightarrow \pi_{q-\!1}(SO_{p})$$
where 
$F_{p,q}: \pi_{p-1}(SO_{q})\otimes \pi_{q}(S^{p}) \longrightarrow \pi_{q-1}(SO_{p})$
is the bilinear map from (\ref{eq: bundle composition}). 
The $\pi_{q}(S^{p})$-pair $\mathbf{G}_{p,q}$ together with the maps
$\partial_{q}: \pi_{q}(S^{p}) \rightarrow \pi_{q-1}(SO_{p})$ and $\bar{\pi}_{q}: \pi_{q-1}(SO_{p}) \rightarrow \pi_{q}(S^{p})$
from (\ref{eq: form parameter maps}) make the $4$-tuple
$(\mb{G}_{p,q}, \; \partial_{q}, \; \pi_{q},\; (-1)^{q})$
into a form parameter for the category $\Ab_{\pi_{q}(S^{p})}^{2}$.
It then follows directly from Lemma \ref{lemma: properties of the invariants} that the $5$-tuple 
\begin{equation} \label{eq: Wall-form of M}
(\mathcal{W}_{p,q}(M), \; \lambda_{p, q}, \; \mu_{q}, \; \alpha_{p}, \; \alpha_{q})
\end{equation}
is a Wall form with form-parameter $(\mb{G}_{p,q}, \; \partial_{q}, \; \pi_{q}, (-1)^{q})$. We call the Wall form of (\ref{eq: Wall-form of M}) the \textit{Wall form of degree $(p,q)$ associated to $M$}. 

Let $N$ be another compact $m$-dimensional manifold with $\kappa(N) > q - p -1$.
Let $f: N \longrightarrow M$ be a smooth embedding. The map on homotopy groups induced by $f$ commutes with $\tau_{p,q}$ and thus $f$ induces an $\pi_{q}(S^{p})$-map which we denote by 
$f_{*}: \mathcal{W}_{p, q}(N) \longrightarrow \mathcal{W}_{p, q}(M).$
It follows directly from Proposition \ref{prop: functoriality of the structure} that the $\pi_{q}(S^{p})$-map $f_{*}$ is a morphism of Wall forms. 

\subsection{The Standard Wall Form.} \label{subsection: standard nonsingular wall form}
Fix a finitely generated Abelian group $H$. 
We first define an $H$-pair 
$\mb{W} := \PH{0}\oplus \PH{1}.$
For $g \in \N$ we denote by $\mb{W}^{g}$ the $g$-fold direct-sum $\mb{W}^{\oplus g}$. We let $\mb{W}$ denote the $H$-pair $\mb{W}^{1}$ and $\mb{W}^{0}$ is understood to be the zero $H$-pair $\mb{0}$. 
Recall,
$$\PH{0}_{-} = 0, \quad \PH{0}_{+} = \Z, \quad \PH{1}_{-} = \Z, \quad \PH{1}_{+} = H.$$
Fix elements 
$a \in \mb{W}_{-}$ and $b \in \mb{W}_{+}$
which correspond to $1 \in \PH{1}_{-} = \Z$ and $1 \in \PH{0}_{+} = \Z$ respectively. 
For $g \in \N$, we denote by 
\begin{equation} \label{eq: standard generators}
 a_{i} \in \mb{W}^{g}_{-} \quad \text{and} \quad b_{i} \in \mb{W}^{g}_{+} \quad \text{for $i = 1, \dots, g$} 
 \end{equation}
the elements  that correspond to the elements $a$ and $b$ coming from the $i$th direct-summand of $\mb{W}$ in $\mb{W}^{g}$.
Now fix a form-parameter $(\mb{G}, \partial, \pi, \epsilon)$. We endow $\mb{W}^{g}$ with the structure of a Wall form with parameter $(\mb{G}, \partial, \pi, \epsilon)$ by setting:
\begin{equation} \label{eq: standard Wall-form maps}
\lambda(a_{i}, b_{j}) = \delta_{i,j}, \quad \mu(b_{i}, b_{j}) = 0, \quad \alpha_{-}(a_{i}) = 0, \quad \alpha_{+}(b_{i}) = 0 \quad \text{for $i, j = 1, \dots g$.}
\end{equation}
These values together with the conditions imposed from Definition \ref{defn: Wall-form} determine the maps $\lambda, \mu,$ and $\alpha_{\pm}$ completely, thus  $(\mb{W}^{g}, \lambda, \mu, \alpha_{-}, \alpha_{+})$
is a Wall form with parameter $(\mb{G}, \partial, \pi, \epsilon)$. We call this the \textit{standard Wall form of rank $g$} with parameter $(\mb{G}, \pi, \partial, \epsilon)$. 
\begin{remark} The standard Wall forms $\mb{W}^{g}$ are defined in the same way for any choice of form-parameter. Notice that the map $\alpha_{-}: \mb{W}_{-} \rightarrow \mb{G}_{-}$ is identically zero 
and $\alpha_{+}: \mb{W}_{+} \rightarrow \mb{G}_{+}$ is determined entirely by $\mu$. For this reason we do not include the form-parameter in the notation of $\mb{W}^{g}$. The relevant form-parameter will always be clear from the context. 
\end{remark}
\begin{proposition} \label{prop: non-singularity} The Wall form $\mb{W}^{g}$ is non-singular as defined in Definition \ref{defn: non-singular}.
\end{proposition}
\begin{proof} We first prove that $\mb{W}$ is non-singular. 
We will show that the maps 
$$\xymatrix{
T^{(0)}: \mb{W}_{-} \longrightarrow \Hom_{\Ab_{H}^{2}}(\mb{W}, \PH{0}), \quad T^{(1)}: \mb{W}_{+} \longrightarrow \Hom_{\Ab_{H}^{2}}(\mb{W}, \PH{1}) }
$$
are isomorphisms. We prove surjectivity first. Let $f: \mb{W} \longrightarrow \PH{0}$ be an $H$-map. 
The $H$-map $f$ is determined entirely by where it sends the elements $a \in \mb{W}_{-}$ and $b \in \mb{W}_{+}$. Since $\PH{0}_{-} = \mb{0}$, $f_{-}(a)$ is automatically zero.
Let $t$ denote the element $f_{+}(b) \in \Z = \PH{0}_{+}$.
Since $\lambda(a, b) = 1$, it follows immediately from the definition of $T^{(0)}$ that $f = T^{(0)}(t\cdot a)$. 
This proves that $T^{(0)}$ is surjective. 

Now let $\varphi: \mb{W} \longrightarrow \PH{1}$ be an $H$-map. 
We let $s$ and $h$ denote the elements 
$$\varphi_{-}(a) \in \Z = \PH{1}_{-} \quad  \text{and} \quad \varphi_{+}(b) \in H = \PH{1}_{+}$$ 
respectively. 
We then let $w$ denote the sum, $s\cdot b + \epsilon\cdot\tau(a, h)$ (where $\epsilon$ is value $+1$ or $-1$ specified in the chosen 
 form-parameter). 
Using $\lambda(a, \tau(a, h)) = 0$ and $\mu(b, b) = 0$, we compute
$$ \lambda(a,\; w) = \lambda(a,\; s\cdot b) = s, $$
 and 
$$\mu(b, w)  = \mu(b,\; \epsilon\cdot\tau(a, h)) =  \epsilon\cdot\mu(\epsilon\cdot\tau(a, h),\; b) = \lambda(a,\; b)\cdot h = h.$$
It follows that $\varphi = T^{(1)}(w)$. This proves that $T^{(1)}$ is surjective. 
Injectivity of $T^{(0)}$ and $T^{(1)}$ follows immediately from the fact that $\lambda(a, b) = 1$ and $\mu(\tau(a, h), b) = 1$. 
This proves that $\mb{W}$ is non-singular. 
Non-singularity of $\mb{W}^{g}$ for $g \in \N$ follows from the canonical isomorphisms,
$$\xymatrix{
\Hom_{\Ab_{H}^{2}}(\mb{W}^{g}, \PH{\nu}) \cong \Hom_{\Ab_{H}^{2}}(\mb{W}, \PH{\nu})^{\oplus g} } \quad \text{for $\nu = 0, 1$.}$$
This concludes the proof of the proposition. 
\end{proof}
Recall from Notation \ref{notation: orthogonal sum notation} that $\mb{N}\perp\mb{N}'$ denotes the sum of orthogonal sub-Wall forms.
\begin{proposition} \label{prop: orthog comp} Let $f: \mb{W} \longrightarrow \mb{M}$ be a morphism of Wall forms. Then $f$ is split-injective and there is an orthogonal direct-sum decomposition, 
$f(\mb{W})^{\perp}\perp f(\mb{W}) = \mb{M}$
\end{proposition}
\begin{proof}
Let $v$ and $w$ denote the elements $f_{-}(a) \in \mb{M}_{-}$ and $f_{+}(b) \in \mb{M}_{+}$ respectively. 
Consider the $H$-map,
$$ T^{(0)}(v)\oplus T^{(1)}(w): \mb{M} \; \longrightarrow \; \PH{0}\oplus\PH{1} = \mb{W}.$$
It follows from the definition of $T^{(0)}$ and $T^{(1)}$ that the kernel of $T^{(0)}(v)\oplus T^{(1)}(w)$ is $f(\mb{W})^{\perp}$. 
The morphism $f: \mb{W} \rightarrow \mb{M}$ (which is determined entirely by $a \mapsto v$ and $b \mapsto w$)
yields a section of the $H$-map $T^{(0)}(v)\oplus T^{(1)}(w)$. This section yields a splitting $f(\mb{W})\perp f(\mb{W})^{\perp} = \mb{M}$. It follows from the existence of this section that $f$ is injective. 
\end{proof}

We will study Wall forms by probing them by morphisms from $\mb{W}^{g}$. The next proposition is immediate from the definition of $\mb{W}^{g}$.
\begin{proposition} \label{prop: construction of morphisms}
Let $\mb{M}$ be a Wall form. Let $x_{1}, \dots, x_{k} \in \mb{M}_{-}$ and $y_{1}, \dots, y_{k} \in \mb{M}_{+}$ be elements such that,
$\lambda(x_{i}, y_{j}) = \delta_{i,j},$ $\mu(y_{i}, y_{j}) = 0,$ $\alpha_{-}(x_{i}) = 0,$ and $\alpha_{+}(y_{i}) = 0$
for $i, j = 1, \dots, k$.
Then there exists a unique morphism $f: \mb{W}^{k} \longrightarrow \mb{M}$ such that $f_{-}(a_{i}) = x_{i}$ and $f_{+}(b_{i}) = y_{i}$ for $i = 1, \dots, k$. 
\end{proposition}

We now define the \textit{rank} of a Wall form.
\begin{defn} \label{defn: rank of wall form}
For a Wall form $\mb{M}$, the \text{rank} is defined to be the non-negative integer
$$r(\mb{M}) := \max\{ \; g \in \N \; | \; \text{there exists a morphism $\mb{W}^{g} \rightarrow \mb{M}$}\}.$$
The \textit{stable rank} of $\mb{M}$ is then defined to be the non-negative integer
$$\bar{r}(\mb{M}) := \max\{\; \mb{r}(\mb{M}\oplus \mb{W}^{g}) - g \; | \; g \in \N\; \}.$$
 \end{defn}

We have an immediate corollary to Proposition \ref{prop: orthog comp}.
\begin{corollary} \label{cor: rank reduction} Let $f: \mb{W}^{g} \rightarrow \mb{M}$ be a morphism. Then
$\bar{r}(f(\mb{W}^{g})^{\perp}) \geq \bar{r}(\mb{M}) - g.$
\end{corollary}
\begin{proof} 
This follows immediately from the orthogonal splitting in Proposition \ref{prop: orthog comp} and the definition of the stable rank.
\end{proof}

We now define a simplicial complex which is the subject of the main result of the next section.
\begin{defn} \label{defn: algebraic simplicial complex} For a Wall form $\mb{M}$ let $L(\mb{M})$ be the simplicial complex whose vertices are given by morphisms $f: \mb{W} \rightarrow \mb{M}$. A set of vertices $\{f_{0}, \dots, f_{l}\}$ is an $l$-simplex if the sub Wall forms, $f_{0}(\mb{W}), \dots, f_{l}(\mb{W}) \leq \mb{M}$ are pairwise orthogonal. 
\end{defn}

The next two propositions let us deduce certain algebraic properties of a Wall form $\mb{M}$ from topological properties of the simplicial complex $L(\mb{M})$. The propositions are analogous to \cite[Propositions 3.3 and 3.4]{GRW 14} and are proven in exactly the same way, so we omit the proofs.
\begin{proposition}[Transitivity] \label{lemma: transitivity} Let
$f_{1}, f_{2}: \mb{W} \rightarrow
\mb{M}$ be morphisms. If $|L(\mb{M})|$ is
connected then there exists an automorphism $\Phi:
\mb{M} \longrightarrow \mb{M}$ such that $f_{1}
= \Phi\circ f_{2}$.
\end{proposition}

\begin{proposition}[Cancelation] \label{lemma: cancelation} Let
$\mb{M}$ and $\mb{N}$ be Wall forms such that there
exists an isomorphism
$\mb{M}\oplus\mb{W} \stackrel{\cong} \longrightarrow
\mb{N}\oplus\mb{W}.$
If $|L(\mb{M}\oplus\mb{W})|$ is connected then there
exists an isomorphism of Wall forms,
$\mb{M} \stackrel{\cong} \longrightarrow \mb{N}.$
\end{proposition}

\section{High-Connectivity of $L(\mb{M})$} \label{section: connectivity of L}
For what follows, fix once and for all a finitely generated Abelian group $H$ and a form-parameter $(\mb{G}, \partial, \pi, \epsilon)$ for the category $\Ab_{H}^{2}$. Let 
$d$ denote the integer $d(H)$ (which recall, is the generating set length for the group $H$).
Let $\mb{M}$ be a Wall form with form-parameter $(\mb{G}, \partial, \pi, \epsilon)$.
Our main result is the following theorem:
\begin{theorem} \label{thm: high connectivity} Suppose that
$\bar{r}(\mb{M}) \geq g$. Then $lCM(L(\mb{M})) \geq
\frac{1}{2}(g -1-d)$ and the geometric realization
$|L(\mb{M})|$ is $\frac{1}{2}(g - 4 - d)$ connected.
\end{theorem}
\subsection{Some Technical Results}
In order to prove Theorem \ref{thm: high connectivity}, 
we will need some technical algebraic results regarding the standard Wall form $\mb{W}^{g}$. The key results in this subsection are Lemmas \ref{lemma: automorphism 2} and \ref{cor: automorphism 1}. We will need to fix some notation. 
For $g \in \N$ denote by 
\begin{equation} \mb{W}^{g}_{(0)} \leq \mb{W}^{g} \quad  \text{and} \quad \mb{W}_{(1)}^{g} \leq \mb{W}^{g} \end{equation}
the sub-$H$-pairs which correspond to the direct summands of $(\PH{0})^{\oplus g}$ and $(\PH{1})^{\oplus g}$ respectively in $\mb{W}^{g}$. 
We have,
\begin{equation} \label{eq: isotropic decomp}
(\mb{W}_{(0)}^{g})_{-} = 0, \quad (\mb{W}_{(0)}^{g})_{+} =  \Z\langle b_{1}, \dots, b_{g}\rangle, \quad (\mb{W}_{(1)}^{g})_{-} = \mb{W}_{-}^{g} =  \Z\langle a_{1}, \dots, a_{g}\rangle.
\end{equation}
Furthermore, $\tau$ maps  the tensor product
$\mb{W}^{g}_{-}\otimes H$ isomorphically onto  $(\mb{W}^{g}_{(1)})_{+}$. We denote 
\begin{equation} \label{eq: restriction of codomain}
\tau^{(1)}: \mb{W}^{g}_{-}\otimes H \stackrel{\cong} \longrightarrow (\mb{W}^{g}_{(1)})_{+}; \quad x\otimes h \mapsto \tau(x, h)
\end{equation}
the isomorphism defined by restricting the the codomain of $\tau$ to $(\mb{W}^{g}_{(1)})_{+}$.
\begin{proposition} \label{prop: endomorphism} For all $g \in \N$ any endomorphism $f: \mb{W}^{g} \longrightarrow \mb{W}^{g}$ is an automorphism. 
\end{proposition}
\begin{proof}
Let $f: \mb{W}^{g} \longrightarrow \mb{W}^{g}$ be given. By Proposition \ref{prop: orthog comp} $f$ is split-injective. To prove the result it will suffice to show that $f$ is surjective. 
By Proposition \ref{prop: orthog comp} there is an orthogonal splitting $f(\mb{W}^{g})\perp f(\mb{W}^{g})^{\perp} \; = \; \mb{W}^{g}$. 
Since the morphism $f$ is an isomorphism onto its image $f(\mb{W}^{g})$, the above splitting induces an isomorphism (of $H$-pairs), $\mb{W}^{g}\oplus f(\mb{W}^{g})^{\perp} \cong \mb{W}^{g}$. 
Since $\mb{W}^{g}$ is finitely generated as an $H$-pair, it follows from Proposition \ref{prop: direct sum iso} that $f(\mb{W}^{g})^{\perp} = \mb{0}$, and thus $f(\mb{W}^{g}) = \mb{W}^{g}$. 
This concludes the proof of the proposition.
\end{proof}

\begin{lemma} \label{eq: rank reduction 1} Let $k$ and $g$ be positive integers and 
 let $f: \mb{W}^{k} \longrightarrow \mb{W}^{g+k}$ be a morphism of Wall forms. Then there is an isomorphism 
 $\mb{W}^{g} \stackrel{\cong} \longrightarrow f(\mb{W}^{k})^{\perp}$. 
\end{lemma}
\begin{proof} 
We prove the result by induction on $k$ with $g \in \N$ arbitrary.
For the case where $k = 1$ (which is the base-case of the induction),
let $f: \mb{W} \longrightarrow \mb{W}^{g+1}$ be given.
To prove that $f(\mb{W})^{\perp} \cong \mb{W}^{g}$,
by Proposition \ref{prop: endomorphism} it will suffice to construct a morphism $\varphi: \mb{W}^{g} \longrightarrow f(\mb{W})^{\perp}$. 
Assuming such a morphism $\varphi$ exists we write $\mb{W}^{g+1} = \mb{W}\perp\mb{W}^{g}$. The direct sum of the morphisms $f$ and $\varphi$ yield the morphism,
 $$\xymatrix{
 f\perp \varphi:  \mb{W}^{g+1} = \mb{W}\perp\mb{W}^{g} \ar[rr] &&  f(\mb{W})\perp f(\mb{W})^{\perp} = \mb{W}^{g+1}
 }$$
which is an isomorphism by Proposition \ref{prop: endomorphism}. Since $f$ is an isomorphism onto its image $f(\mb{W})$, it follows that $\varphi$ must be an isomorphism as well. 
 
We let $v$ and $w$ denote the elements  $f_{-}(a) \in \mb{W}^{g+1}_{-}$ and $f_{+}(b) \in \mb{W}^{g+1}_{+}$ respectively, where $a$ and $b$ are the standard generators of $\mb{W}$. 
 Using the direct-sum decomposition $\mb{W}^{g+1} = \mb{W}^{g+1}_{(0)}\oplus\mb{W}^{g+1}_{(1)}$ (direct-sum as $H$-pairs),
we may write $w = w_{(0)} + w_{(1)}$ for unique choice of $w_{(0)} \in (\mb{W}^{g+1}_{(0)})_{+}$ and $w_{(1)} \in (\mb{W}^{g+1}_{(1)})_{+}$. Since 
$\lambda(v, w) = \lambda(v, w_{(0)}) = 1,$ 
it follows that the subgroups
$\Z\langle v \rangle \leq \mb{W}^{g+1}_{-}$ and  $\Z\langle w_{(0)} \rangle \leq (\mb{W}^{g+1}_{(0)})_{+}$
split as direct summands in $\mb{W}^{g+1}_{-}$ and $(\mb{W}^{g+1}_{(0)})_{+}$ respectively.
Thus for $i = 1, \dots, g$ we may choose elements $x_{i} \in \mb{W}^{g+1}_{-}$ and $y_{i} \in (\mb{W}^{g+1}_{(0)})_{+}$ such that 
$$
\lambda(x_{i}, y_{j}) = \delta_{i, j}, \quad \lambda(x_{i}, w) = 0, \quad \text{and} \quad  \lambda(v, y_{i}) = 0 \quad \text{for all $i, j \in \{1, \dots, g\}$.}
$$
Notice that since $y_{i} \in (\mb{W}^{g+1}_{(0)})_{+}$, it is automatic that $\mu(y_{i}, y_{j}) = 0$ and $\mu(w_{(0)}, y_{i}) = 0$ for all $i, j$.
For each $i$, we set
$$h_{i} \; := \; \mu(y_{i}, w) = \mu(y_{i}, w_{(1)}) \quad \text{and} \quad \hat{y}_{i} := y_{i} - \tau(v, h_{i}).$$
Since 
$$\mu(\hat{y}_{i}, w) = \mu(y_{i}, w) - \mu(\tau(v, h_{i}), w) = h_{i} - \lambda(v, w)\cdot h_{i} = h_{i} - h_{i} = 0,$$
(the second equality follows by Condition ii. of Definition \ref{defn: Wall-form}) we have $\hat{y}_{i} \in f(\mb{W})^{\perp}_{+}$. We have $x_{i} \in f(\mb{W})^{\perp}_{-}$ as well and so 
the lemma will be proven once we construct a morphism $\varphi: \mb{W}^{g} \rightarrow \mb{W}^{g+1}$ such that $\varphi_{-}(a_{i}) = x_{i}$ and $\varphi_{+}(b_{i}) = \hat{y_{i}}$ for $i = 1, \dots, g$. The existence of such a morphism will follow from Proposition \ref{prop: construction of morphisms} once we demonstrate that,
\begin{equation} \label{eq: basis elements}
\alpha_{-}(x_{i}) = 0, \quad \alpha_{+}(\hat{y}_{i}) = 0, \quad \text{and} \quad  \lambda(x_{i}, \hat{y}_{j}) = \delta_{i,j} \quad \text{for $i, j = 1, \dots, g$.}
\end{equation}
First, $\alpha_{-}(x_{i}) = 0$ is automatic since $\alpha_{-}: \mb{W}^{g+1}_{-} \rightarrow \mb{G}_{-}$ is by definition the zero map. 
Condition iv. of  Definition \ref{defn: Wall-form} implies that 
$$
\alpha_{+}(\hat{y}_{i}) = \alpha_{+}(y_{i}) - \alpha_{+}(\tau(v, h_{i})) + \partial[\mu(y_{i}, \tau(v, h_{i}))],
$$
and conditions vi. and ii. imply that,
$$\alpha_{+}(\tau(v, h_{i})) = \tau_{\mb{G}}(\alpha_{-}(v), h_{i}) = 0 \quad \text{and} \quad \partial[\mu(y_{i}, \tau(v, h_{i}))] = \partial(\lambda(v, y_{i})\cdot h_{i}) = \partial(0) = 0.$$ 
Combining these above calculations yields
$\alpha_{+}(\hat{y}_{i}) = \alpha_{+}(y_{i})$
and since $y \in (\mb{W}^{g}_{(0)})_{+}$, we have $\alpha_{+}(\hat{y}_{i}) = 0$.
Finally, we compute
$$\lambda(x_{i}, \hat{y}_{j}) = \lambda(x_{i}, y_{j}) + \lambda(x_{i}, \tau(v, h_{j})) = \delta_{i,j} + 0 = \delta_{i,j},$$
where the second equality follows from condition i. of Definition \ref{defn: Wall-form} and the fact that the $x_{i}$ and $y_{i}$ were initially chosen so that $\lambda(x_{i}, y_{j}) = \delta_{i,j}$.
We have established (\ref{eq: basis elements}) and thus by Proposition \ref{prop: construction of morphisms} there exists a morphism $\varphi: \mb{W} \rightarrow f(\mb{W})^{\perp}$ with $\varphi_{-}(a_{i}) = x_{i}$ and $\varphi_{+}(b_{i}) = \hat{y}_{i}$ for $i = 1, \dots, g$. 
This established the  result for the case where $k = 1$.

For the inductive step, assume that the result holds for the $k-1$-case (with $g$ arbitrary) and let $f: \mb{W}^{k} \longrightarrow \mb{W}^{g+k}$ be a morphism. 
Consider the orthogonal direct-sum decomposition $\mb{W}^{k} = \mb{W}^{1}\perp\mb{W}^{k-1}$ and let denote $\mb{V}$ dneote the sub-Wall form $f(\mb{W}^{1})^{\perp} \leq \mb{W}^{g+k}$. 
Since $f$ preserves orthogonality we have $f(\mb{W}^{k-1}) \leq \mb{V}$.
The base case of the induction implies that there is an isomorphism, 
$\mb{V} \cong \mb{W}^{g+k-1}$. 
 Using the isomorphism $\mb{V} \cong \mb{W}^{g+k-1}$, the induction assumption implies that there is an isomorphism 
$$f(\mb{W}^{k-1})^{\perp}\cap\mb{V} \cong \mb{W}^{g}.$$
Now, for any Wall form $\mb{M}$ and two orthogonal sub-Wall forms $\mb{A}, \mb{B} \leq \mb{M}$ we have, 
$$(\mb{A}\perp\mb{B})^{\perp} = \mb{A}^{\perp}\cap\mb{B}^{\perp}.$$
From this it follows that,
$$f(\mb{W}^{k}) = [f(\mb{W}^{1})\perp f(\mb{W}^{k-1})]^{\perp} = \mb{V}\cap f(\mb{W}^{k-1})^{\perp} \cong \mb{W}^{g}.$$
This concludes the inductive step and the proof of the lemma.
\end{proof}

\begin{proposition} \label{prop: morphism 1} Let $k$ and $g$ be positive integers such that $k+ 1 \leq g$. Let $y \in (\mb{W}^{g}_{(0)})_{+}$ and $x_{1}, \dots, x_{k} \in \mb{W}^{g}_{-}$. Then there exists a morphism $f: \mb{W}^{k+1} \longrightarrow \mb{W}^{g}$ such that $y \in f(\mb{W}^{k+1})_{+}$ and $x_{i} \in f(\mb{W}^{k+1})_{-}$ for $i = 1, \dots, k$. 
\end{proposition}
\begin{proof} We prove this first for the case that $y = 0$; we prove the proposition for this $y = 0$ case by induction on $k$. Let $x \in \mb{W}^{g}_{-}$. Since $\mb{W}^{g}_{-}$ is a free $\Z$-module there exists a \textit{primitive} element $\bar{x} \in \mb{W}^{g}_{-}$ and $t \in \Z$ such that $t\cdot \bar{x} = x$. By primitive we mean that the submodule $\Z\langle\bar{x}\rangle \leq \mb{W}^{g}_{-}$ splits as a direct summand.
By non-singularity of the Wall form structure, it follows that there exists $w \in (\mb{W}^{g}_{(0)})_{+}$ such that $\lambda(\bar{x}, w) = 1$. Since $\alpha_{+}(w) = 0$, it follows that there exists a morphism $f: \mb{W} \rightarrow \mb{W}^{g}$ such that $f_{-}(a) = \bar{x}$ and $f_{+}(b) = w$. It follows that $x \in f(\mb{W})$. 
This proves the base case. 

Now let $x_{1}, \dots, x_{k} \in \mb{W}^{g}_{-}$ be given and suppose that the statement holds for the $(k-1)$-case. There exists a morphism $f': \mb{W}^{k-1} \longrightarrow \mb{W}^{g}$ such that $x_{i} \in f'(\mb{W}^{k-1})_{-}$ for $i = 1, \dots, k-1$. 
Using the orthogonal splitting $f'(\mb{W}^{k-1})\perp f'(\mb{W}^{k-1})^{\perp} = \mb{W}^{g}$, we may write $x_{k} = x_{k}' + x_{k}''$ for $x_{k}' \in f'(\mb{W}^{k-1})_{-}$ and $x_{k}'' \in f'(\mb{W}^{k-1})^{\perp}_{-}$. 
By Lemma \ref{eq: rank reduction 1} there is an isomorphism $f'(\mb{W}^{k-1})^{\perp} \cong \mb{W}^{g-k+1}$. 
Using this isomorphism, by what was proven in the base case of the induction there exists a morphism $f'': \mb{W} \rightarrow f'(\mb{W}^{k-1})^{\perp}$ with $x_{k}'' \in f''(\mb{W})_{-}$. 
Writing $\mb{W}^{k} = \mb{W}^{k-1}\perp\mb{W}$, the morphism given by sum of $f'$ and $f''$,
$$\xymatrix{
f'\perp f'': \mb{W}^{k} =  \mb{W}^{k-1}\perp\mb{W} \ar[rr] && f'(\mb{W}^{k-1})\perp f'(\mb{W}^{k-1})^{\perp} = \mb{W}^{g}
}$$ 
is the desired morphism. This proves the proposition for the case when $y = 0$. 

Now let 
 $y, x_{1}, \dots, x_{k}$ be given and suppose that $y \neq 0$. $(\mb{W}^{g}_{(0)})_{+}$ is a free $\Z$-module of rank $g$ and 
thus there exists a primitive element $\bar{y} \in (\mb{W}^{g}_{(0)})_{+}$ and $t \in \Z$ such that $y = t\cdot \bar{y}$. 
As before, by non-singularity of the Wall form structure there exists $v \in \mb{W}^{g}_{-}$ such that $\lambda(v, \bar{y}) = 1$. 
Since $\alpha_{+}(\bar{y}) = 0$, there exists a morphism $\phi: \mb{W} \rightarrow \mb{W}^{g}$ such that $\phi_{-}(a) = v$ and $\phi_{+}(b) = \bar{y}$ and thus $y \in \phi_{+}(\mb{W})$. 
Now, let $\mb{V}$ denote the sub-Wall form $\phi(\mb{W}) \leq \mb{W}^{g}$. 
We have a splitting $\mb{W}^{g} = \mb{V}\perp\mb{V}^{\perp}$. 
Using this splitting we may write $x_{i} = x_{i}' + x_{i}''$ for $x' \in \mb{V}_{-}$ and $x_{i}'' \in \mb{V}^{\perp}_{-}$ for $i = 1, \dots, k$. 
By Lemma \ref{eq: rank reduction 1} there is an isomorphism $\mb{V}^{\perp} \cong \mb{W}^{g-1}$. 
Using this isomorphism, by what was proven in the previous paragraph there exists a morphism $\psi: \mb{W}^{k} \longrightarrow \mb{V}^{\perp}$ such that $x_{i}'' \in \psi(\mb{W}^{k})_{-}$ for $i = 1, \dots, k$. The sum $\phi\perp\psi: \mb{W}\perp\mb{W}^{k} \longrightarrow \mb{V}\perp\mb{V}^{\perp} = \mb{W}^{g}$ then yields the desired morphism. This concludes the proof of the proposition. 
\end{proof}

We now arrive at the first result where the \textit{generating set-length} of $H$, $d(H)$, plays a role.
\begin{lemma} \label{lemma: automorphism 2} Let $g$ be an integer such that $d = d(H) \leq g-1$. 
Let $y \in \mb{W}^{g}_{+}$. 
Then there exists an automorphism 
$\Phi: \mb{W}^{g} \longrightarrow \mb{W}^{g}$ such that $\Phi^{-1}_{+}(y) \in (\mb{W}^{d +1}\oplus\mb{0})_{+} \leq \mb{W}^{g}_{+}$. 
\end{lemma}  
\begin{proof} It will suffice to construct a morphism $f: \mb{W}^{d+1} \longrightarrow \mb{W}^{g}$ with $y \in f(\mb{W}^{d+1})_{+}$.
With this morphism constructed, by Lemma \ref{eq: rank reduction 1} there exists an isomorphism $\varphi: \mb{W}^{g-d-1} \stackrel{\cong} \longrightarrow f(\mb{W}^{d+1})^{\perp}$. 
Then the map 
$\xymatrix{
f\perp\varphi: \mb{W}^{g} = \mb{W}^{d+1}\perp\mb{W}^{g-d-1} \ar[rr] && f(\mb{W}^{d+1})\perp f(\mb{W}^{d+1})^{\perp}
}$
yields the desired automorphism. 

Using the splitting $\mb{W}^{g}_{(0)}\oplus \mb{W}^{g}_{(1)} = \mb{W}^{g}$, we may write $y = y_{(0)} + y_{(1)}$ with $y_{(0)} \in  (\mb{W}^{g}_{(0)})_{+}$ and $y_{(1)} \in  (\mb{W}^{g}_{(1)})_{+}$.
Let $\{h_{1}, \dots, h_{d}\}$ be a generating set for $H$ of minimal length. For $i = 1, \dots, d$, we denote by $H_{i} \leq H$ the subgroup generated by the element $h_{i}$. Since the generating set $\{h_{1}, \dots, h_{d}\}$ is of minimal length
we have a direct sum decomposition,  $H = \bigoplus_{i=1}^{d}H_{i}$. 
Consider the tensor product, $\mb{W}^{g}_{-}\otimes H$. 
The natural isomorphism, 
$\mb{W}^{g}_{-}\otimes H\; \cong \; \bigoplus_{i=1}^{d}(\mb{W}^{g}_{-}\otimes H_{i})$ 
implies that
for any element $\hat{v} \in \mb{W}^{g}_{-}\otimes H$, there exists $v_{1}, \dots, v_{d} \in \mb{W}^{g}_{-}$ such that $\hat{v} = \sum_{i= 1}^{d}v_{i}\otimes h_{i}$.
Using this general fact about the representability of elements in the tensor product $\mb{W}^{g}_{-}\otimes H$,
the isomorphism 
$\tau^{(1)}: \mb{W}^{g}_{-}\otimes H \stackrel{\cong} \longrightarrow (\mb{W}^{g}_{(1)})_{+}$ of (\ref{eq: restriction of codomain})
implies that there exists 
elements $x_{1}, \dots, x_{d} \in \mb{W}^{g}_{-}$ such that,  
$$y_{(1)} = \sum_{i=1}^{d}\tau(x_{i}, h_{i}).$$
By Proposition \ref{prop: morphism 1} there exists a morphism $f: \mb{W}^{d+1} \longrightarrow \mb{W}^{g}$ such that $y_{(0)} \in f(\mb{W}^{d+1})_{+}$ and $x_{i} \in f(\mb{W}^{d+1})_{-}$ for $i = 1, \dots, d$. 
It follows that $\tau(x_{i}, h_{i}) \in f(\mb{W}^{d+1})_{+}$ for $i = 1, \dots, d$ and thus $y_{(1)} \in f(\mb{W}^{d+1})_{+}$ and $y = y_{(0)} + y_{(1)} \in f(\mb{W}^{d+1})_{+}$.
 This concludes the proof of the lemma.
\end{proof}

\begin{lemma} \label{cor: automorphism 1} Let $x \in \mb{W}^{g}_{-}$. Then there exists an automorphism $\Phi: \mb{W}^{g} \longrightarrow \mb{W}^{g}$ such that 
$\Phi^{-1}_{-}(x) \in (\mb{W}^{1}\oplus\mb{0})_{-} \leq \mb{W}^{g}_{-}$. 
\end{lemma}
\begin{proof} By Proposition \ref{prop: morphism 1} there exists a morphism $f: \mb{W} \longrightarrow \mb{W}^{g}$ such that $x \in f(\mb{W})_{-}$. Lemma \ref{eq: rank reduction 1} yields an isomorphism $\varphi: \mb{W}^{g-1} \stackrel{\cong} \longrightarrow f(\mb{W})^{\perp}$. The sum of $f$ and $\varphi$
$$\xymatrix{
f\perp\varphi: \mb{W}^{g} = \mb{W}\perp\mb{W}^{g-1} \ar[rr] && f(\mb{W})\perp f(\mb{W})^{\perp} = \mb{W}^{g}
}$$ 
yields the desired automorphism. 
\end{proof}

The following two Corollaries are the key results used in the proof of Theorem \ref{thm: high connectivity}. Compare to \cite[Corollary 4.2]{GRW 14}.
\begin{corollary} \label{prop: linear functional 1} Let $g$ be an integer such that $g - 1 \geq d = d(H)$. Let $\mb{M}$ be a Wall form with $r(\mb{M}) \geq g$. Let $\varphi: \mb{M} \longrightarrow \PH{1}$ be an $H$-map. Then $r(\kernel(\varphi)) \geq g - d -1$. Similarly, if $\bar{r}(\mb{M}) \geq g$, then $\bar{r}(\kernel(\varphi)) \geq g - 1 - d$.
\end{corollary}
\begin{proof} Since $r(M) \geq g$, there exits a morphism $f: \mb{W}^{g} \longrightarrow \mb{M}$. 
Denote by 
$\bar{\varphi}: \mb{W}^{g} \longrightarrow \PH{1}$
the $H$-map given by the composition
$\mb{W}^{g} \stackrel{f} \longrightarrow \mb{M} \stackrel{\varphi} \longrightarrow \PH{1}.$
Since $\mb{W}^{g}$ is a non-singular Wall form, there exists $x \in \mb{W}^{g}_{+}$ so that $T^{(1)}(x): \mb{W}^{g} \longrightarrow \PH{1}$ is equal to $\bar{\varphi}$. 
Now, by Lemma \ref{lemma: automorphism 2} there is a Wall form automorphism $\Phi: \mb{W}^{g} \longrightarrow \mb{W}^{g}$  such that  $\Phi_{+}^{-1}(x) \in [\mb{W}^{d+1}\oplus\mathbf{0}]_{+}$. It follows that the composition, 
$\xymatrix{
\mb{W}^{g} \ar[r]^{  \Phi} & \mb{W}^{g} \ar[r]^{ f} & \ar[r]^{ \varphi} \mb{M} & \PH{1}
}$
is equal to $T^{(1)}(x')$ for $x' \in (\mb{W}^{d+1}\oplus\mathbf{0})_{+}$ where $x' = \Phi^{-1}_{+}(x)$. 
From this we see that $f\circ\Phi$ maps $(\mathbf{0}\oplus \mb{W}^{g- d -1})$ into $\kernel(\varphi) \leq \mb{M}$ and thus $r(\kernel(\varphi)) \geq g -d-1$. 

Now let $\bar{r}(\mb{M}) \geq g$ and $\varphi: \mb{M} \longrightarrow \PH{1}$ be an $H$-map. Let $j \geq 0$ be an integer such that $r(\mb{M}\oplus\mb{W}^{j}) \geq g$. We then consider the $H$-map given by
$\xymatrix{
\mb{M}\oplus\mb{W}^{j} \ar[r]^{ \ \ \ \pr} & \mb{M} \ar[r]^{\varphi} & \PH{1}.
}$  
By the previous paragraph, $r(\kernel(\varphi\circ\pr)) \geq g -1$. By construction, $\kernel(\varphi\circ\pr) \subset \kernel(\varphi)\oplus\mb{W}^{j}$. It follows that $\bar{r}(\kernel{\varphi})\geq g -d -1$.
\end{proof}
The next Corollary is similar the previous but for an $H$-map $\varphi: \mb{M} \longrightarrow \PH{0}$. In contrast to the previous Corollary, the following does not depend on the generating set-length $d(H)$.
\begin{corollary} \label{cor: linear functional 2} Let $\mb{M}$ be a Wall form with $r(\mb{M}) \geq g$. Let $\varphi: \mb{M} \longrightarrow \PH{0}$ be an $H$-map. Then $r(\kernel(\varphi)) \geq g -1$. Similarly, if $\bar{r}(\mb{M}) \geq g$, then $\bar{r}(\kernel(\varphi)) \geq g - 1$.
\end{corollary}
\begin{proof}
Since $r(M) \geq g$, there exists a morphism  $f: \mb{W}^{g} \longrightarrow \mb{M}$. 
Denote by 
$\bar{\varphi}: \mb{W}^{g} \longrightarrow \PH{0}$
the $H$-map given by the composition
$\mb{W}^{g} \stackrel{f} \longrightarrow \mb{M} \stackrel{\varphi} \longrightarrow \PH{0}.$
By non-singularity there exists $x \in \mb{W}^{g}_{+}$ so that $T^{(1)}(x) = \bar{\varphi}$. 
By Corollary \ref{cor: automorphism 1} there is an automorphism $\Phi: \mb{W}^{g} \longrightarrow \mb{W}^{g}$  such that  $\Phi_{+}^{-1}(x) \in [\mb{W}^{1}\oplus\mathbf{0}]_{+}$. The proof then follows the same argument as in the proof of Corollary \ref{prop: linear functional 1} but with $d = 0$. 
\end{proof}

\begin{corollary} \label{cor: rank reduction 3} Let $\mb{M}$ be a Wall form and let $\mb{N} \leq \mb{M}$ be a sub-Wall form with $r(\mb{N}) \geq g \geq d = d(H)$. Let $f: \mb{W} \rightarrow \mb{M}$ be a morphism.  
Then $r(\mb{N}\cap f(\mb{W})^{\perp}) \geq g -2 - d.$ Similarly, if $\bar{r}(\mb{N}) \geq g$ then $\bar{r}(\mb{N}\cap f(\mb{W})^{\perp}) \geq g -2 -d$.
\end{corollary}
\begin{proof} We set $v := f_{-}(a)$ and $w = f_{+}(b)$. Consider the $H$-maps 
$T^{(0)}(v): \mb{M} \longrightarrow \PH{0}$ and $T^{(1)}(w): \mb{M} \longrightarrow \PH{1}.$
The Wall form given by $\mb{N}\cap f(\mb{W})^{\perp}$ is equal to the kernel of the $H$-map,
$[T^{(0)}(v)\oplus T^{(1)}(w)]|_{\mb{N}}: \mb{N} \longrightarrow \PH{0}\oplus\PH{1}.$
The result then follows by applying Corollaries \ref{prop: linear functional 1} and \ref{cor: linear functional 2}, one after the other. 
\end{proof}

\subsection{Proof of Theorem \ref{thm: high connectivity}} The proof of Theorem \ref{thm: high connectivity} is by an induction argument which is essentially parallel to the proof of
\cite[Theorem 3.2]{GRW 14}. The key algebraic results about Wall forms needed are Corollaries \ref{prop: linear functional 1}, \ref{cor: linear functional 2}, and \ref{cor: rank reduction 3}.
The zero-step of this induction argument is established in Proposition \ref{prop: zero case}. 
Throughout this section the abelian group $H$ and the form-parameter $(\mb{G}, \partial, \pi, \epsilon)$ are fixed. We let $d$ denote the generating set length $d(H)$.
 \begin{proposition} \label{prop: zero case} If
$\bar{r}(\mb{M}) \geq 2 + d$, then $L(\mb{M}) \neq
\emptyset$, and if $\bar{r}(\mb{M}) \geq 4 + d$ then
$L(\mb{M})$ is connected.
\end{proposition}
\begin{proof} We first prove that if $r(\mb{M}) \geq 4 + d$ then $L(\mb{M})$ is connected. So, let $r(\mb{M}) \geq 4 + d$. It follows that there exists some morphism $h_{0}: \mb{W} \rightarrow \mb{M}$ with $r(h_{0}(\mb{W})^{\perp}) \geq 3 + d$. 
Now let $h: \mb{W} \rightarrow \mb{M}$ be a morphism. We will show that there exists a path in $L(\mb{M})$ connecting $h$ to $h_{0}$. By Corollary \ref{cor: rank reduction 3}, we have 
$r(h_{0}(\mb{W})^{\perp}\cap h(\mb{W})^{\perp}) \geq 1$
and so there exists a morphism $h': \mb{W} \rightarrow h_{0}(\mb{W})^{\perp}\cap h(\mb{W})^{\perp}.$ It follows that both $\{h', h_{0}\}$ and $\{h', h\}$ are $1$-simplices in $L(\mb{M})$ and so there is a path of length $2$ in $L(\mb{M})$ connecting $h$ to $h_{0}$.

Now suppose that $\bar{r}(\mb{M}) \geq 4 + d$. Then there exists a Wall form $\mb{N}$ with $r(\mb{N}) \geq 4 + d$ and integer $k \geq 0$ such that $\mb{M}\oplus\mb{W}^{k}\cong \mb{N}\oplus\mb{W}^{k}$. 
By the argument given in the above paragraph, $L(\mb{N}\oplus\mb{W}^{j})$ is connected for all $j \geq 0$. We then may apply Proposition \ref{lemma: cancelation} inductively to deduce that $\mb{M}\cong \mb{N}$ as Wall forms. It follows then that $r(\mb{M}) \geq 4 + d$ and that $L(\mb{M})$ is connected. 

Suppose that $\bar{r}(\mb{M}) \geq 2 + d$. There exists a Wall form $\mb{N}$ with $r(\mb{N}) \geq 2 + d$ and integer $k \geq 0$ such that $\mb{M}\oplus\mb{W}^{k}\cong \mb{N}\oplus\mb{W}^{k}$. As in the previous paragraph we may apply Proposition \ref{lemma: cancelation} inductively to 
see that $\mb{M}\oplus\mb{W}$ is isomorphic to $\mb{N}\oplus\mb{W}$. Let $\varphi: \mb{M}\oplus\mb{W} \stackrel{\cong} \longrightarrow \mb{N}\oplus\mb{W}$ be such an isomorphism. Then the Wall form $\mb{M}$ is isomorphic to the sub-Wall form given by $\varphi(\mb{0}\oplus\mb{W})^{\perp} \leq \mb{N}\oplus\mb{W}$. By Corollary \ref{cor: rank reduction 3} we have, $r(\varphi(\mb{0}\oplus\mb{W})^{\perp}) \geq r(\mb{N}\oplus\mb{W}) - 2 - d$. Since $r(\mb{N}\oplus\mb{W}) \geq 3 + d$ we have $r(\mb{M}) \geq 1$.
It follows that there exists a morphism $\mb{W} \rightarrow \mb{M}$ and thus $L(\mb{M})$ is non-empty. This concludes the proof of the proposition. 
\end{proof}

\begin{proof}[Proof of Theorem \ref{thm: high connectivity}.] We prove this by induction on $g$. To check the base case of the induction, let $g = 4 + d$ and $\bar{r}(\mb{M}) \geq 4 + d$. 
Then by Proposition \ref{prop: zero case}, $L(\mb{M})$ is connected and hence is $0 = \frac{1}{2}[(4- d) - 4- d]$-connected. 
We need to also show that $lCM(L(\mb{M})) \geq \frac{1}{2}[(4 + d) - 1- d] = \frac{3}{2}$. 
It will suffice to show that the link of any $0$-simplex is $(\frac{3}{2} - 0 - 2) = -\frac{1}{2}$-connected, or in other words non-empty. 
Denote by $\sigma$ the $0$-simplex represented by the morphism $f: \mb{W} \rightarrow \mb{M}$. 
It follows immediately from the definition of the link of a simplex that there is an isomorphism of simplicial complexes,  $Lk_{L(\mb{M})}(\sigma) \cong L(f(\mb{W})^{\perp})$. 
By Corollary \ref{cor: rank reduction} we have $\bar{r}(f(\mb{W})^{\perp}) \geq 3 + d$. 
It follows from Proposition \ref{prop: zero case} that $L(\mb{M})$ is indeed non-empty. 

Now suppose that the theorem holds for the $g-1$ case. 
Let $\mb{M}$ be a Wall form with $\bar{r}(\mb{M}) \geq g$ and $g \geq 4+d$. By Proposition \ref{prop: zero case} there exists a morphism $f: \mb{W} \longrightarrow \mb{M}$ and by Corollary \ref{cor: rank reduction} we have $\bar{r}(f(\mb{W})^{\perp}) \geq g -1$. 
We let $\mb{M}'$ denote the sub-Wall form $f(\mb{W})^{\perp} \leq \mb{M}$ and consider the sub-Wall form 
$$\mb{M}'\perp f(\mb{W}_{(1)}) \leq \mb{M}.$$
The chain of inclusions
$$\mb{M}' \longrightarrow \mb{M}'\perp f(\mb{W}_{(1)}) \longrightarrow \mb{M}$$
 induces a chain of embeddings of simplicial complexes,
\begin{equation}
\xymatrix{ 
L(\mb{M}') \ar[rr]^{i_{1}} && L(\mb{M}'\perp f(\mb{W}_{(1)})) \ar[rr]^{i_{2}} && L(\mb{M}).
}
\end{equation}
The composition $i_{2}\circ i_{1}$ is 
 null-homotopic since the vertex in $L(\mb{M})$ determined by the morphism $f: \mb{W} \rightarrow \mb{M}$ is adjacent to every simplex in the subcomplex $L(\mb{M}') \leq L(\mb{M})$. We will then apply Proposition \ref{prop: high connected map} to the maps $i_{1}$ and $i_{2}$ with $n := \frac{1}{2}(g -4-d)$ to conclude that the composition $i_{2}\circ i_{1}$ is $n$-connected. This together with the fact that $i_{2}\circ i_{1}$ is null-homotopic will imply that the space $|L(\mb{M})|$ is $n  = \frac{1}{2}(g -4-d)$-connected. 

So, let $\sigma$ be an $l$-simplex of $L(\mb{M}'\perp f(\mb{W}_{(1)}))$. Since $f(\mb{W}_{(1)})$ has trivial Wall form structure, i.e. $\lambda, \mu,$ and $\alpha_{\pm}$ are identically zero on $\mb{W}_{(1)}$, the projection map $\pi: \mb{M}'\perp f(\mb{W}_{(1)}) \rightarrow \mb{M}'$ (which is an $H$-map) is a morphism of Wall forms. 
This induces a map of simplicial complexes 
$\bar{\pi}: L(\mb{M}'\perp f(\mb{W}_{(1)})) \longrightarrow L(\mb{M}')$
such that $\bar{\pi}\circ i_{1} = Id_{L(\mb{M}')}$.
There is an isomorphism of simplicial complexes,
$$[Lk_{L(\mb{M}'\perp f(\mb{W}_{(1)}))}(\sigma)]\cap L(\mb{M}') \; = \; Lk_{L(\mb{M}')}(\bar{\pi}(\sigma)).$$
By the induction assumption we have that this is 
$\frac{1}{2}(g-2-d) - l -2 =  (n -l -1)$-connected.
Lemma \ref{prop: high connected map} then implies that the map $i_{1}$ is $n$-connected. 

We now prove that $i_{2}$ is $n$-connected. We set $x := f_{-}(a)$ where $a \in \mb{W}_{-}$ is the standard generator. 
Since $\mb{M}' = f(\mb{W})^{\perp}$ and $\lambda(a, \tau(a, h))  = 0$ for all $h \in H$, it follows that the sub-Wall form $(\mb{M}'\perp f(\mb{W}_{(1)}) \leq \mb{M}$ is precisely the kernel of the $H$-map 
$$T^{(0)}(x): \mb{M} \longrightarrow \PH{0}.$$
Recall that by Corollary \ref{cor: linear functional 2}, passing to the kernel of such an $H$-map decreases the stable-rank of the Wall form by at most $1$.
Now let $\sigma = \{f_{0}, \dots, f_{l}\}$ be an $l$-simplex in $L(\mb{M})$. The pairwise orthogonality of the morphisms $f_{0}, \dots, f_{l}$ implies that there is an isomorphism $\sum_{i=0}^{l}f_{i}(\mb{W}) \cong \mb{W}^{l+1}$. Let
$\mb{M}''$ denote the sub-Wall form given by the orthogonal complement, $(\sum_{i=0}^{l}f_{i}(\mb{W}))^{\perp}$. 
We have $\bar{r}(\mb{M}'') \geq g - l -1$ and there is a simplicial isomorphism
$Lk_{L(\mb{M})}(\sigma) \cong L(\mb{M}'')$.
We have a simplicial isomorphism,
$$Lk_{L(\mb{M})}(\sigma)\cap L(\mb{M}'\perp f(\mb{W}_{(1)})) \; \cong \; L(\mb{M}''\cap [\mb{M}'\perp f(\mb{W}_{(1)})]).$$
Since $\mb{M}'\perp f(\mb{W}_{(1)}) = \ker(T^{(0)}(x))$, it follows that 
the intersection, $\mb{M}''\cap(\mb{M}'\perp f(\mb{W}_{(1)}))$ is equal to the kernel of the $H$-map, 
$T^{(0)}(x)|_{\mb{M}''}: \mb{M}'' \longrightarrow \PH{0}.$
It follows from Corollary \ref{cor: linear functional 2} that,  
$$\bar{r}[\mb{M}''\cap(\mb{M}'\perp f(\mb{W}_{(1)}))] \; = \; \bar{r}[\kernel(T^{(0)}(x)|_{\mb{M}''})] \; \geq \; g - l - 2.$$
 The induction assumption implies that the geometric realization
$$|Lk_{L(\mb{M})}(\sigma)\cap L(\mb{M}'\perp f(\mb{W}_{(1)}))|$$ 
is 
$\frac{1}{2}(g - l - 2 - 4 - d)$-connected. 
Using the inequality, 
$$\frac{1}{2}(g - l - 2 - 4 - d) \; \geq \; \frac{1}{2}(g - 4- d) - l - 1 \; = \; n - l -1,$$
we then apply Proposition \ref{prop: high connected map} again to show that the map $i_{2}$ is $n = \frac{1}{2}(g-4-d)$-connected. 
It follows that the composition $i_{2}\circ i_{1}$ is $\frac{1}{2}(g-4- d)$-connected. This together with the fact that $i_{2}\circ i_{1}$ is null-homotopic implies that $|L(\mb{M})|$ is $\frac{1}{2}(g-4-d)$-connected. 

We still need to verify that $lCM(L(\mb{M})) \geq \tfrac{1}{2}(g-1-d)$. 
Let $\sigma = \{f_{0}, \dots, f_{l}\} \leq L(\mb{M})$ be an $l$-simplex. Let $\mb{M}'$ denote the sub-Wall form given by the orthogonal complement, $(\sum_{i=0}^{l}f_{i}(\mb{W}))^{\perp}$. 
We have $\bar{r}(\mb{M}') \geq g - l -1$ and $Lk_{L(\mb{M})}(\sigma) \cong L(\mb{M}')$. By the induction assumption, $|L(\mb{M}')|$ is $\frac{1}{2}(g- l -1 - 4- d)$-connected. 
The inequality,
$$\tfrac{1}{2}(g -l-1 -4-d) \; = \; \tfrac{1}{2}(g - 1 - 4-d) - l \; = \;  \tfrac{1}{2}(g - 1-d) - l - 2$$
then implies that $lCM(L(\mb{M})) \geq \frac{1}{2}(g-1-d)$. 
This concludes the proof of the theorem. 
\end{proof}

\section{High Connectivity of $K(M)_{p,q}$} \label{section: high connectivity of K}
Let $M$ be
a compact $m$-dimensional manifold with non-empty boundary. Let $p$ and $q$ be positive integers with $p + q = m$ such that 
$p < q < 2p - 2$ and $q - p + 1 < \kappa(M).$
Throughout this section we will denote, $d := d(\pi_{q}(S^{p}))$.
Theorem \ref{thm: high connectivity of K(M)} states that if $r_{p,q}(M) \geq g$, then the simplicial complex 
 $K(M)_{p,q}$ of Definition \ref{defn: the embedding complex} is $\frac{1}{2}(g - 4-d)$-connected. We prove this result 
by comparing the complex $K(M)_{p,q}$ to the simplicial complex $L(\mathcal{W}_{p,q}(M))$ where $\mathcal{W}_{p,q}(M)$ is the Wall form of degree $(p,q)$ associated to $M$ as defined in Section \ref{subsection: Wall-form of a manifold}.
 The form-parameter associated to $\mathcal{W}_{p,q}(M)$ is given by $(\mathbf{G}_{p,q}, \partial_{q}, \pi_{q}, \epsilon)$
 where recall, $\mathbf{G}_{p,q}$ is the $\pi_{q}(S^{p})$-pair given by the data,
 $$(\mathbf{G}_{p,q})_{-} = \pi_{p-1}(SO_{q}), \quad (\mathbf{G}_{p,q})_{+} = \pi_{q}(SO_{p}), \quad \tau = F_{p,q}: \pi_{p-1}(SO_{q})\otimes\pi_{q}(S^{p}) \longrightarrow \pi_{q}(SO_{p}).$$
Throughout this section  $\mb{W}^{k}$
is the standard Wall form from Section \ref{subsection: standard nonsingular wall form} with form-parameter $(\mathbf{G}_{p,q}, \partial_{q}, \pi_{q}, \epsilon).$ 
The same form-parameter is fixed throughout the whole section. 

Recall the manifolds $W^{k}_{p,q}$ and the elements $\sigma_{1}, \dots, \sigma_{k} \in \pi_{p}(W_{p,q}^{k})$ and $\rho_{1}, \dots, \rho_{k} \in \pi_{q}(W_{p,q})$ specified in Section \ref{subsection: Application to Products of Spheres}.  
For each $k \in \N$ there is an isomorphism of Wall forms 
\begin{equation} \label{eq: chosen iso 1}
\mathbf{W}^{k} \stackrel{\cong} \longrightarrow \mathcal{W}(W^{k}_{p,q})
\end{equation}
defined by sending $a_{i}$ to $\sigma_{i}$ and $b_{i}$ to $\rho_{i}$ for $i = 1, \dots, k$ where $a_{i}$ and $b_{i}$ are the standard generators of $\mb{W}^{k}$.
It follows from Proposition \ref{eq: values of the invariants 2} that this is an isomorphism. The next proposition follows immediately from this isomorphism.
\begin{proposition} \label{pro: geo to algebraic rank} Let be $M$ be a manifold such that $r_{p,q}(M) \geq k$. Then $r(\mathcal{W}_{p,q}(M)) \geq k$. 
\end{proposition}

Recall the manifold $\bar{W}_{p,q}$ used in the definition of the simplicial complex $K(M)_{p, q}$. The inclusion $W_{p,q} \hookrightarrow \bar{W}_{p,q}$ 
is a homotopy equivalence and since it is also a co-dimension zero embedding, Proposition \ref{prop: functoriality of the structure} implies that it induces an 
isomorphism of Wall forms,
$\mathcal{W}_{p,q}(W_{p,q}) \cong  \mathcal{W}_{p,q}(\bar{W}_{p,q}).$ 
For an embedding $\phi: \bar{W}_{p,q} \rightarrow M$, we denote by 
$$(\phi|_{W_{p,q}})_{*}: \mathcal{W}_{p,q}(W_{p,q}) \longrightarrow \mathcal{W}_{p,q}(M)$$
the morphism induced by the restriction $\phi|_{W_{p,q}}: W_{p,q} \rightarrow M$. If $\phi_{0}, \phi_{1}: \bar{W}_{p,q} \rightarrow M$ are embeddings with disjoint cores, i.e. $\phi_{0}(C_{p,q})\cap \phi_{1}(C_{p,q}) = \emptyset$, it follows from Proposition \ref{prop: Generalized Whitney Trick} that the images of the induced morphisms, 
$(\phi_{0}|_{W_{p,q}})_{*}, \; (\phi_{1}|_{W_{p,q}})_{*}: \mathcal{W}_{p,q}(W_{p,q}) \longrightarrow \mathcal{W}_{p,q}(M)$
are orthogonal as sub-Wall forms.  

We consider a simplicial map 
\begin{equation} \label{eq: complex map}
\Psi: K(M)_{p, q} \longrightarrow L(\mathcal{W}_{p, q}(M))
\end{equation}
defined by sending a vertex $(\phi, t)$ in $K(M)_{p,q}$
 to the morphism of Wall forms 
$\mb{W} \longrightarrow \mathcal{W}_{p,q}(M)$
given by precomposing, $(\phi|_{W_{p,q}})_{*}: \mathcal{W}_{p,q}(W_{p,q}) \longrightarrow \mathcal{W}_{p,q}(M)$ with the isomorphism (\ref{eq: chosen iso 1}) for $k = 1$. 
Condition ii. of Definition \ref{defn: the embedding complex}  implies that $\Psi$ preserves adjacencies between vertices and thus determines a simplicial map. 

\begin{proof}[Proof of Theorem \ref{thm: high connectivity of K(M)}] 
Suppose that $r_{p,q}(M) \geq g$. We will prove that $|K(M)_{p,q}|$ is $\frac{1}{2}(g - 4- d)$-connected.
Let $k \leq \frac{1}{2}(g-4-d)$ and consider a map $f: S^{k} \longrightarrow |K(M)_{p,q}|$, which we may assume is simplicial with respect to some PL triangulation of $S^{k} = \partial I^{k+1}$. By Theorem \ref{thm: high connectivity} the composition $\partial I^{k+1} \longrightarrow |K(M)_{p,q}| \longrightarrow |L(\mathcal{W}_{p,q}(M))|$ is null-homotopic and so extends to a map $F: I^{k+1} \longrightarrow |L(\mathcal{W}_{p,q}(M))|$, which we may suppose is simplicial with respect to a PL triangulation of $I^{k+1}$ extending the triangulation of its boundary. By 
Theorem \ref{thm: high connectivity} we have $lCM(L(\mathcal{W}_{p,q}(M))) \geq
\frac{1}{2}(g-1-d)$. Using Theorem \ref{thm: cohen mac trick}, as $k
+1 \leq \frac{1}{2}(g-1-d)$, we can arrange that $F$ is simplexwise
injective on the interior of $I^{k+1}$. We choose a total order on the
interior vertices of $I^{k+1}$ and we will now inductively choose lifts of each
vertex to $K(M)_{p,q}$ which respect adjacencies.

Let $h: \mb{W} \rightarrow \mathcal{W}_{p,q}(M)$ represent a vertex 
 in the image of the interior of $I^{k+1}$. 
We have 
$$\alpha_{q}(h_{+}(b)) = 0, \quad \alpha_{p}(h_{-}(a)) = 0, \quad \text{and} \quad \lambda_{p,q}(h_{-}(a), \; h_{+}(b)) = 1.$$ 
 Let $j_{p}: S^{p} \rightarrow M, \quad j_{q}: S^{q} \rightarrow M$
 be embeddings that represent 
 $$h_{-}(a) \in \pi_{p}(M) = \mathcal{W}_{p, q}(W)_{-} \quad \text{and} \quad h_{+}(b) \in \pi_{q}(M) = \mathcal{W}_{p, q}(M)_{+}$$ 
 respectively (the existence of such embeddings follows from Lemma \ref{lemma: represent by embeddings}). Since $\alpha_{-}(a) = 0$ and $\alpha_{+}(b) = 0$, both $j_{p}(S^{p}) \subset M$ and $j_{q}(S^{q}) \subset M$ have trivial normal bundles. 
 Since $\lambda_{q}(a, b) = 1$, by an application of the \textit{Whitney trick} \cite[Theorem 6.6]{M 65} we may assume that $j_{p}(S^{p})$ and $j_{q}(S^{q})$ intersect  transversely at exactly one point.  
We then find embeddings 
$\bar{j}_{p}: S^{p}\times D^{q} \longrightarrow M$ and $\bar{j}_{q}: D^{p}\times S^{q} \longrightarrow M$
of the normal bundles of $j_{p}(S^{p})$ and $j_{q}(S^{q})$ respectively so that  
$$\bar{j}_{p}(S^{p}\times D^{q}) \cap \bar{j}_{q}(D^{p}\times S^{q}) \; = \; U\times V$$
where $U \subset \bar{j}_{p}(S^{p}\times\{0\})$ and $V \subset \bar{j}_{q}(\{0\}\times S^{q})$ are submanifolds (embedded as closed subsets), diffeomorphic to the disks $D^{p}$ and $D^{q}$ respectively. The push-out 
\begin{equation} \label{eq: pushout} \bar{j}_{p}(S^{p}\times D^{q}) \bigcup_{U\times V} \bar{j}_{q}(D^{p}\times S^{q}) \end{equation}
is diffeomorphic to the manifold $W_{p, q} = S^{p}\times S^{q}\setminus \Int(D^{m})$, after smoothing the corners. The inclusion map of this push-out 
determines an embedding $\phi: W_{p,q} \hookrightarrow M$. We need to extend $\phi$ to an embedding $\bar{\phi}: \bar{W}_{p,q} \rightarrow M$ which satisfies condition (i) of Definition \ref{defn: the embedding complex}. 
We choose an embedding
$\gamma: I \rightarrow M\setminus\text{Int}(\phi(W_{p,q}))$
with $\gamma(0) \in \partial M$ and $\gamma(1) \in \partial[\phi(W_{p,q})]$. 
Consider the subspace 
$(I\times \{0\})\cup(\{0, 1\}\times D^{m-1}) \subset I\times D^{m-1}.$
The embedding $\gamma$ extends to an embedding 
$$\gamma': (I\times \{0\})\cup(\{0, 1\}\times D^{m-1}) \longrightarrow M\setminus\text{Int}(\phi(W_{p,q}))$$ 
with $\gamma'(\{0\}\times D^{m-1}) \subset \partial M$ and $\gamma'(\{1\}\times D^{m-1}) \subset \partial \phi(W_{p,q})$. Since $M\setminus\text{Int}(\phi(W_{p,q}))$ is simply connected, $\gamma'$ extends to embedding $\bar{\gamma}: I\times D^{m-1} \rightarrow M\setminus\text{Int}(\phi(W_{p,q}))$. 
The embedding $\bar{\gamma}$ together with $\phi$ combine to give an embedding 
$\bar{\phi}: \bar{W}_{p,q} \rightarrow M$ which satisfies the boundary condition in the definition of $K(M)_{p,q}$. Thus $\bar{\phi}$ determines a vertex in $K(M)_{p,q}$ which maps to $h: \mathbf{W} \rightarrow \mathcal{W}_{p,q}(M)$ under $\Psi$. 

Now let $h: \mb{W} \rightarrow \mathcal{W}_{p,q}(M)$ represent an interior vertex and let $h^{1}, \dots, h^{k}$ be the vertices adjacent to $h$ that have already been lifted and denote by $(\phi_{1}, t_{1}), \dots, (\phi_{k}, t_{k})$ their lifts. We have for $i = 1, \dots, k$:
$$\lambda_{p, q}(h_{-}(a), \;  h^{i}_{+}(b)) = 0, \quad \lambda_{p,q}(h_{-}^{i}(a), \; h_{+}(b)) = 0, \quad \text{and} \quad \mu_{p, q}(h_{+}(b), \; h_{+}^{i}(b)) = 0.$$
By inductive application of the classical Whitney trick \cite[Theorem 6.6]{M 65} we 
may choose an embedding $j_{p}: S^{p} \longrightarrow M$ which represents the class $h_{-}(e)$, such that $j_{p}(S^{p})$ is disjoint from the \textit{cores} $\phi_{i}(C_{p,q})$,  for $i = 1, \dots, k$. Similarly, by applying 
Proposition \ref{prop: Generalized Whitney Trick} and the Whitney-Trick inductively, we may choose an embedding $j_{q}: S^{q} \rightarrow M$ which represents the class $h_{+}(b) \in \pi_{q}(M)$, such that $j_{q}(S^{q})$ is disjoint form the \textit{cores} $\phi_{i}(C_{p,q})$ and such that $j_{q}(S^{q})$ intersects $j_{p}(S^{p})$ transversally at exactly one point. 
 We then carry out the plumbing construction employed in the zeroth step of the induction on the embeddings $j_{p}$ and $j_{q}$ to obtain an embedding 
$\phi: \bar{W}_{p,q} \longrightarrow M$ such that $\phi(C_{p,q}) \cap \phi_{i}(C_{p,q}) = \emptyset$ for $i \in 1, \dots, k$. This completes the induction and concludes the proof of Theorem \ref{thm: high connectivity of K(M)}.
\end{proof}

\section{Homological Stability} \label{section: homological stability}
With our main technical result Theorem \ref{thm: high connectivity of K(M)} established, in this final section we show how Theorem \ref{thm: high connectivity of K(M)} implies our main result which is Theorem \ref{thm: Main Theorem 1}.

\subsection{A Model for $\BDiff^{\partial}(M)$.} \label{subsection: semi-simplicial resolution}
Let $M$ be a compact manifold of dimension $m$ with non-empty boundary.  
We now construct a concrete model for $\BDiff^{\partial}(M)$. 
Fix an embedding of a collar,
 $h: [0,\infty)\times\partial M \longrightarrow M \quad \text{with $h^{-1}(\partial M) = \{0\}\times\partial M$.}$
 Fix once and for all an embedding, $\theta: \partial M \longrightarrow \R^{\infty}$
 and let $S$ denote the submanifold $\theta(\partial M) \subset \R^{\infty}$.
\begin{defn} \label{defn: moduli M} 
 We define
$\mathcal{M}(M)$ to be the set of compact $m$-dimensional submanifolds 
$M' \subset [0,\infty)\times\R^{\infty}$
such that:
\begin{enumerate}
\item[i.] $M'\cap(\{0\}\times\R^{\infty}) = S$ and $M'$ contains $[0, \epsilon)\times S$ for some $\epsilon > 0$.
\item[ii.] The boundary of $M'$ is precisely $\{0\}\times S$.
\item[iii.] $M'$ is diffeomorphic to $M$ relative to $S$. 
\end{enumerate}
Denote by $\mathcal{E}(M)$ the space, in the $C^{\infty}$-topology, of embeddings $\psi: M \rightarrow [0,\infty)\times\R^{\infty}$
for which there exists $\epsilon > 0$ such that $\psi\circ h(t, x) = (t, \theta(x))$ for all $(t, x) \in [0,\epsilon)\times\partial M$, where $h: [0,\infty)\times\partial M \rightarrow M$ is the chosen collar embedding. The space $\mathcal{M}(M)$ is topologized as a quotient of the space $\mathcal{E}(M)$ where two embeddings are identified if they have the same image. 
\end{defn}

It follows from Definition \ref{defn: moduli M} that $\mathcal{M}(M)$ is equal to the orbit space, $\mathcal{E}(M)/\Diff^{\partial}(M)$.
By the main result of \cite{BF 81},  the quotient map, $\mathcal{E}(M) \longrightarrow \mathcal{E}(M)/\Diff^{\partial}(M) = \mathcal{M}(M)$
is a locally trivial fibre-bundle. This, together with the fact that $\mathcal{E}(M)$ is weakly contractible implies that there is a weak-homotopy equivalence,
$
\mathcal{M}(M) \sim \BDiff^{\partial}(M).
$

Now let $p$ and $q$ be positive integers with $p + q = m$ and $p \leq q$. 
Recall from Section \ref{subsection: statement of results} the manifold $V_{p,q}$, given by forming the connected sum of $[0,1]\times\partial M$ with $S^{p}\times S^{q}$. 
Choose a collared embedding 
$$\alpha: V_{p,q} \longrightarrow [0,1]\times\R^{\infty}$$ 
such that,
for $(i, x) \in \{0,1\}\times\partial M \subset V_{p,q}$, the equation $\alpha(i, x) = (i, \theta(x))$ is satisfied.
For any submanifold $M' \subset [0,\infty)\times\R^{\infty}$, denote by 
$M' + e_{1} \subset [1,\infty)\times\R^{\infty}$
the submanifold obtained by linearly translating $M'$ over $1$-unit in the first coordinate. Then for $M' \in \mathcal{M}(M)$, the submanifold $\alpha(V_{p,q}) \cup (M'\cup e_{1}) \subset [0,\infty)\times\R^{\infty}$ is an element of $\mathcal{M}(M\cup_{\partial M}V_{p,q})$. Thus, we have a continuous map,
\begin{equation} \label{ref: p-stabilization map}
s_{p,q}: \mathcal{M}(M) \longrightarrow \mathcal{M}(M\cup_{\partial M}V_{p,q}); \quad V \mapsto \alpha(V_{p,q})\cup(V + e_{1}).
\end{equation}
\begin{remark} \label{remark: stabilization map}
The construction of $s_{p,q}$ depends on the choice of embedding $\alpha: V_{p,q} \rightarrow [0,1]\times\R^{\infty}$. However, any two such embeddings are isotopic (the space of all such embeddings is weakly contractible). It follows that the homotopy class of $s_{p,q}$ does not depend on any of the choices made. In this way the manifold $V_{p,q}$ determines a unique homotopy class of maps, $\BDiff^{\partial}(M) \longrightarrow \BDiff^{\partial}(M\cup_{\partial M}V_{p,q})$ which is in the same homotopy class as the map (\ref{eq: classifying space stabilization}) used in the statement of Theorem \ref{subsection: statement of results}.  
\end{remark}

\subsection{A Semi-Simplicial Resolution}
Let $M$ be as in Section \ref{subsection: semi-simplicial resolution}. 
We now construct, for each $p \leq q$ with $p+q = m$, a semi-simplicial space $X_{\bullet}(M)_{p,q}$, equipped with an augmentation $\epsilon_{p,q}: X_{\bullet}(M)_{p,q} \longrightarrow \mathcal{M}(M)$ such that the induced map $|X_{\bullet}(M)_{p,q}| \longrightarrow \mathcal{M}(M)$ is highly connected. 
Such an augmented semi-simplicial space is called a \textit{semi-simplicial resolution}. 
To construct such a semi-simplicial resolution we will need to first define an auxiliary semi-simplicial space related to the complex $K(M)_{p,q}$. 
Let $M$, $p$, and $q$ be as in the previous section and let $a: [0,1)\times\R^{m-1} \longrightarrow M$ be the same embedding used in Definition \ref{defn: the embedding complex}. We define two semi-simplicial spaces $K_{\bullet}(M, a)_{p,q}$ and $\bar{K}_{\bullet}(M, a)_{p,q}$.
\begin{defn} \label{defn: the embedding complex (semi simp)}  The spaces of $l$-simplices, $K_{l}(M, a)_{p,q}$ are defined as follows:
\begin{enumerate} 
\item[(i)] 
The space of $0$-simplices $K_{0}(M, a)_{p,q}$ is defined to have the same underlying set as the set of vertices of the simplicial complex $K(M, a)_{p,q}$. That is, $K_{0}(M, a)_{p,q}$ is the space of pairs 
 $(t, \phi)$, where $t \in \R$ and 
$\phi: \bar{W}_{p,q} \longrightarrow M$ is an embedding which satisfies the condition as in part (i) of Definition \ref{defn: the embedding complex}.
\item[(ii)] The space of $l$-simplices, $K_{l}(W, a) \subset (K_{0}(W, a))^{l+1}$ consists of the ordered $l+1$-tuples 
$((t_{0}, \phi_{0}), \dots, (t_{p}, \phi_{p}))$
such that $t_{0} < \cdots < t_{p}$ and $\phi_{i}(C_{p,q})\cap\phi_{j}(C_{p,q}) = \emptyset$ when $i \neq j$.
\end{enumerate}
The spaces $K_{l}(W, a)\subset (\R\times
  \Emb(\bar{W}_{p,q}, M))^{l+1}$ are topologized using the $C^{\infty}$-topology on
the space of embeddings. 
The assignments $[l] \mapsto K_{l}(M, a)$ define a semi-simplicial space denoted by 
$K_{\bullet}(M, a)_{p,q}$.
 
 Finally, 
$\bar{K}_{\bullet}(M, a)_{p,q} \subset K_{\bullet}(M, a)_{p,q}$ is defined to be the sub-semi-simplicial space consisting of all simplices $((\phi_{0}, t_{0}), \dots, (\phi_{l}, t_{l})) \in K_{l}(M, a)$ such that the intersections $\phi_{i}(\bar{W}_{p,q}) \cap \phi_{j}(\bar{W}_{p,q}) = \emptyset$ whenever $i \neq j$.
\end{defn} 
As before we will drop the embedding $a$ and let $K_{\bullet}(M)_{p,q}$ denote the semi-simplicial space $K_{\bullet}(M, a)_{p,q}$.
Similarly we let $\bar{K}_{\bullet}(M)_{p, q}$ denote the semi-simplicial space $\bar{K}_{\bullet}(M, a)_{p,q}$. 
The following result is proven by assembling several results from \cite{GRW 14}. We give an outline of the proof. 
\begin{corollary} \label{prop: high connectivity of simplicial space 1} Let $M$, $p$, and $q$ be as above and let $g$ be such that $r_{p,q}(M) \geq g \geq d = d(\pi_{q}(S^{p}))$. Then $|\bar{K}_{\bullet}(M)_{p,q}|$ 
is $\frac{1}{2}(g-4- d)$-connected. 
\end{corollary}
\begin{proof}[Proof sketch]
Denote by $K^{\delta}_{\bullet}(M)_{p,q}$ the semi-simplicial space defined by setting each $K^{\delta}_{l}(M)_{p,q}$ equal to the space with the same underlying set as $K_{l}(M)_{p,q}$ but equipped with the discrete topology. There is a map $|K^{\delta}(M)_{p,q}| \longrightarrow |K(M)_{p,q}|$ induced by sending an ordered tuple $((\phi_{0}, t_{0}), \dots, (\phi_{l}, t_{l}))$ to its associated unordered set $\{(\phi_{0}, t_{0}), \dots, (\phi_{l}, t_{l})\}$. Any such ordered tuple (in $K^{\delta}_{l}(M)_{p,q}$) is determined by its associated unordered set and thus 
the above map is a homeomorphism. It follows from Theorem \ref{thm: high connectivity of K(M)} that $|K^{\delta}(M)_{p,q}|$ is $\frac{1}{2}(g-4-d)$-connected. Next, we establish $\frac{1}{2}(g-4-d)$-connectivity of
$|K_{\bullet}(M)_{p,q}|$ by comparing $|K_{\bullet}(M)_{p,q}|$ to $|K^{\delta}_{\bullet}(M)_{p,q}|$. This is done by replicating the argument of \cite[Theorem 4.6]{GRW 12} or \cite[Theorem 5.5]{GRW 14}. We then consider the inclusion map $|\bar{K}_{\bullet}(M)_{p,q}| \longrightarrow |K_{\bullet}(M)_{p,q}|$. This map is a weak-homotopy equivalence by \cite[Corollary 5.8]{GRW 14}. From this weak homotopy equivalence it follows that $|\bar{K}_{\bullet}(M)_{p,q}|$ is $\frac{1}{2}(g-4-d)$-connected. 
\end{proof}

We are now ready to construct the resolution $\epsilon_{p,q}: X_{\bullet}(M)_{p,q} \rightarrow \mathcal{M}(M)$. 
Let $\theta: \partial M \hookrightarrow \R^{\infty}$ be the embedding used in the construction of $\mathcal{M}(M)$.
Pick once and for all a coordinate patch $c_{0}: \R^{m-1} \longrightarrow S = \theta(\partial M)$. This choice of coordinate patch induces for any $M' \in \mathcal{M}(M)$, a germ of an embedding
$
[0,1)\times\R^{m-1} \longrightarrow M'
$
as used in the construction of  the semi-simplicial space $\bar{K}_{\bullet}(M')_{p,q}$ from Definition \ref{defn: the embedding complex}. 
\begin{defn} \label{defn: resolution of moduli}
For each non-negative integer $l$, let $X_{l}(M)$ be the set of pairs $(M', \bar{\phi})$ where $M' \in \mathcal{M}(M)$ and $\bar{\phi} \in \bar{K}_{l}(M')_{p,q}$ where $\bar{K}_{l}(M')_{p,q}$ is defined using the embedding germ $[0,1)\times\R^{m-1} \longrightarrow M'$ induced by the chosen coordinate patch $c_{0}: \R^{m-1} \longrightarrow S$. The space $X_{l}(M)$ is topologized as the quotient, $X_{l}(M) = (\mathcal{E}(M)\times \bar{K}_{l}(M)_{p,q})/\Diff^{\partial}(M)$.
The assignments $[l] \mapsto X_{l}(M)_{p,q}$ make $X_{\bullet}(M)_{p,q}$ into a semi-simplicial space where the face maps are induced by the face maps in $\bar{K}_{\bullet}(M)_{p,q}$.

The projection maps
 $X_{l}(M)_{p,q}  \longrightarrow \mathcal{M}(M)$ given by  $(V, \bar{\phi}) \mapsto V$
yield an augmentation map 
$\epsilon_{p,q}: X_{\bullet}(M)_{p,q} \longrightarrow \mathcal{M}(M).$
We denote by $X_{-1}(M)_{p,q}$ the space $\mathcal{M}(M).$ 
\end{defn}
By construction, the projection maps $X_{l}(M)_{p,q}  \rightarrow \mathcal{M}(M)$ are locally trivial fibre-bundles with standard fibre given by $\bar{K}_{l}(M)_{p,q}$. 
From this we have:
\begin{corollary} \label{prop: highly connected resolution}
The map $|\epsilon_{p,q}|: |X_{\bullet}(M)_{p,q}| \longrightarrow \mathcal{M}(M)$ induced by the augmentation is $\frac{1}{2}(r_{p,q}(M) - 2 - d)$-connected. 
\end{corollary} 
\begin{proof}
Since the projection maps $X_{l}(M)_{p,q}  \rightarrow \mathcal{M}(M)$ are locally trivial with fibre $\bar{K}_{l}(M)_{p,q}$, it follows from \cite[Lemma 2]{RW 10} that there is a homotopy-fibre sequence
$|\bar{K}_{\bullet}(M)_{p,q}| \rightarrow |X_{\bullet}(M)_{p,q}| \rightarrow \mathcal{M}(M).$
The result follows from the long-exact sequence on homotopy groups. 
\end{proof}

\subsection{Proof of Theorem \ref{thm: Main Theorem 1}} \label{section: homological stability}
We now will show how to use the \textit{semi-simplicial resolution} $\epsilon_{p,q}: X_{\bullet}(M)_{p,q} \rightarrow \mathcal{M}(M)_{p,q}$ to complete the proof of Theorem \ref{thm: Main Theorem 1}.
First, we fix some new notation which will make the steps of the proof easier to state.
For what follows let $M$ be a compact $m$-dimensional manifold with non-empty boundary. As in the previous sections choose positive integers $p$ and $q$ with $p + q = m$ such that 
$p < q < 2p - 2$ and $q - p + 1 < \kappa(M)$. Let $d$ denote the generating set length $d(\pi_{q}(S^{p}))$. 
We work with the same choice of $p$ and $q$ throughout the entire section.
For each $g \in \N$ we denote by $M_{g}$ the manifold obtained by forming the connected-sum of $M$ with $(S^{p}\times S^{q})^{\# g}$.
Notice that $\partial M = \partial M_{g}$ for all $g \in \N$. 
We consider the spaces $\mathcal{M}(M_{g})$.
For each $g \in \N$, the stabilization map from (\ref{ref: p-stabilization map}) yields a map,
$s_{p,q}: \mathcal{M}(M_{g}) \longrightarrow \mathcal{M}(M_{g+1})$, $M' \mapsto V_{p,q}\cup(M' + e_{1}).$
Using the weak equivalence $\mathcal{M}(M_{g}) \sim \BDiff^{\partial}(M_{g})$ and Remark \ref{remark: stabilization map}, 
Theorem \ref{thm: Main Theorem 1} translates to the following:
\begin{theorem} \label{thm: main theorem neq notation}
The induced map
$(s_{p,q})_{*}: H_{k}(\mathcal{M}(M_{g})) \longrightarrow H_{k}(\mathcal{M}(M_{g+1}))$
is an isomorphism when $k \leq \frac{1}{2}(g - 3 - d)$ and is an epimorphism when $l \leq \frac{1}{2}(g - 1 - d)$. 
\end{theorem}
Theorem \ref{thm: main theorem neq notation} is equivalent to Theorem \ref{thm: Main Theorem 1}. Their equivalence follows from the next proposition. 
\begin{proposition} Let $X$ be a compact manifold of dimension $m$ with non-empty boundary. Suppose that $r_{p,q}(X) \geq g$. Then there exists a compact $m$-dimensional manifold $X_{0}$ with $\partial X_{0} = \partial X$ and a diffeomorphism,
$X \stackrel{\cong} \longrightarrow X_{0}\#(S^{p}\times S^{q})^{\# g}.$
\end{proposition}
\begin{proof} First let $r_{p,q}(X) = 1$. By Theorem \ref{thm: high connectivity of K(M)} there exists an embedding    
$f: \bar{W}_{p,q} \longrightarrow X$ which satisfies the boundary condition from the definition of the simplicial complex $K(X)_{p,q}$. Let $U \subset X$ be a closed neighborhood of the boundary diffeomorphic to $\partial X\times[0,1]$ and such that $U\cap f(\bar{W}_{p,q})$ is diffeomorphic to a disk.
The submanifold $f(\bar{W}_{p,q})\cup U$ is diffeomorphic (after smoothing corners) to the manifold obtained by forming the connected sum of $[0,1]\times\partial X$ with $S^{p}\times S^{q}$. 
We the define $X_{0} := X\setminus\Int[f(\bar{W}_{p,q})\cup U]$. It follows that $X_{0}\#(S^{p}\times S^{q})\cong X$. This proves the proposition in the case that $r_{p,q}(X) = 1$. The general case follows by induction. 
\end{proof}

Since $r(M_{g}) \geq g$ for $g \in \N$, it follows from Corollary \ref{prop: highly connected resolution} that the map 
$$|\epsilon_{p,q}|: |X_{\bullet}(M_{g})_{p,q}| \longrightarrow X_{-1}(M_{g})_{p,q} := \mathcal{M}(M_{g}).$$
 is $\frac{1}{2}(g - 2-d)$-connected. With this established, the proof of Theorem \ref{thm: main theorem neq notation} proceeds in exactly the same way as in \cite[Section 5]{GRW 12}. We provide an outline for how to complete the proof and refer the reader to \cite[Section 5]{GRW 12} for details.
For what follows we fix $g \in \N$. 
For each non-negative integer $k \leq g$ there is a map
\begin{equation} \label{eq: resolution level map}
F_{k}: \mathcal{M}(M_{g-k-1}) \longrightarrow X_{k}(M_{g})_{p,q}
\end{equation}
which is defined in exactly the same way as the map from \cite[Proposition 5.3]{GRW 12}.
From \cite[Proposition 5.3, 5.4 and 5.5]{GRW 12} we have the following.
\begin{proposition} \label{prop: homotopy commutativity} Let $g \geq 4 + d$. We have the following:
\begin{enumerate}
\item[i.] The map $F_{k}: \mathcal{M}(M_{g-k-1}) \longrightarrow X_{k}(M_{g})_{p,q}$ is a weak homotopy equivalence. 
\item[ii.] The following diagram is commutative, 
$$\xymatrix{
\mathcal{M}(M_{g-k-1}) \ar[d]^{F_{k}} \ar[rr]^{S_{p,q}} && \mathcal{M}(M_{g-k}) \ar[d]^{F_{k}} \\
X_{k}(M_{g})_{p,q} \ar[rr]^{d_{k}} && X_{k-1}(M_{g})_{p,q}.
}$$
\item[iii.] The face maps $d_{i}: X_{k}(M_{g})_{p,q} \longrightarrow X_{k-1}(M_{g})_{p,q}$ are weakly homotopic. 
\end{enumerate}
\end{proposition}
\begin{remark} 
The key ingredients needed to prove Proposition \ref{prop: homotopy commutativity} are Proposition \ref{proposition: cancelation} and an analogue of \cite[Corollary 4.4]{GRW 12}. As with Proposition \ref{proposition: cancelation}, the version of \cite[Corollary 4.4]{GRW 12} relevant to our situation is proven in the same way.
\end{remark}

Consider the spectral sequence associated to the augmented semi-simplicial space $X_{\bullet}(M_{g})_{p,q} \rightarrow \mathcal{M}(M_{g})$
with $E^{1}$-term given by 
$E^{1}_{j,l} = H_{j}(X_{l}(M_{g})_{p,q})$ for $l \geq -1$ and $j \geq 0$.
The differential is given by $d^{1} = \sum(-1)^{i}(d_{i})_{*}$, where $(d_{i})_{*}$ is the map on homology induced by the $i$th face map in $X_{\bullet}(M_{g})_{p,q}$. 
The group $E^{\infty}_{j, l}$ is a subquotient of the relative homology group $H_{j+l+1}(X_{-1}(M_{g})_{p,q}, |X_{\bullet}(M_{g})_{p,q}|)$. Proposition \ref{prop: homotopy commutativity} together with Corollary \ref{prop: highly connected resolution} imply the following:
\begin{enumerate}
\item[(a)] For $g \geq 4 + d$, there are isomorphisms $E^{1}_{j,l} \cong H_{l}(\mathcal{M}(M_{g-j-1})).$
\item[(b)] The differential
$d^{1}: H_{l}(\mathcal{M}(M_{g-j-1})) \cong E^{1}_{j,l} \longrightarrow E^{1}_{j-1,l} \cong H_{l}(\mathcal{M}(M_{g-j}))$ is equal to 
$(s_{p,q})_{*}$ when $j$ is even and is equal to zero when $j$ is odd.
\item[(c)] The term $E^{\infty}_{j,l}$ is equal to $0$ when $j +l \leq \frac{1}{2}(g-2-d)$. 
\end{enumerate}
To complete the proof one uses (c) to prove that the differential $d^{1}: E^{1}_{2j,l} \longrightarrow E^{1}_{2j-1,l}$ is an isomorphism 
when $0 < j \leq \frac{1}{2}(g - 3 - d)$  and an epimorphism when $0 < j \leq \frac{1}{2}(g -1 - d)$. 
This is done by carrying out the inductive argument given in \cite[Section 5.2: \textit{Proof of Theorem 1.2}]{GRW 12}. This establishes Theorem \ref{thm: main theorem neq notation} and the main result of this paper, Theorem \ref{thm: Main Theorem 1}.

\end{document}